\documentclass[twoside,a4paper,11pt,limits]{amsart}
\usepackage{graphicx}
\usepackage{color}
\usepackage{amsmath,amssymb,amsthm,dsfont}

\usepackage[hyperref]{paper_diening}
\newtheorem{theorem}{Theorem}[section]
\newtheorem{lemma}[theorem]{Lemma}
\newtheorem{corollary}[theorem]{Corollary}

\newtheorem{remark}[theorem]{Remark}

\numberwithin{equation}{section}

\newcommand{\dt}{\ensuremath{\,{\rm d} t}}

\newcommand{\dx}{\ensuremath{\,{\rm d} x}}
\newcommand{\dy}{\ensuremath{\,{\rm d} y}}

\newcommand{\Acal}{{\mathcal{A}}}
\newcommand{\Adash}{\overline{A}}

\begin{document}

\title[Very weak solutions to nonlinear elliptic systems]{Existence, uniqueness and optimal regularity results for very weak solutions to nonlinear elliptic systems}\thanks{M.~Bul\'{\i}\v{c}ek's work is supported by the project LL1202  financed by the Ministry of Education, Youth and Sports, Czech Republic and by the University Centre for Mathematical Modelling, Applied Analysis and Computational Mathematics (Math~MAC) S.~Schwarzacher thanks the program PRVOUK~P47, financed by Charles University in Prague. M.~Bul\'{\i}\v{c}ek is a  member of the Ne\v{c}as Center for Mathematical Modeling}

\author[M.~Bul\'{\i}\v{c}ek]{Miroslav Bul\'{\i}\v{c}ek} 
\address{Mathematical Institute, Faculty of Mathematics and Physics, Charles University in Prague
Sokolovsk\'{a} 83, 186 75 Prague, Czech Republic}
\email{mbul8060@karlin.mff.cuni.cz}

\author[L.~Diening]{Lars Diening} 
\address{Institut f\"{u}r Mathematik, Universit\"{a}t Osnabr\"{u}ck, Albrechtstr. 28a, 49076 Osnabr\"{u}ck, Germany}
\email{ldiening@uni-osnabrueck.de}

\author[S.~Schwarzacher]{Sebastian Schwarzacher}
\address{Department of Mathematical Analysis, Faculty of Mathematics and Physics,  Charles University in Prague,
Sokolovsk\'{a} 83, 186 75 Prague, Czech Republic}
\email{schwarz@karlin.mff.cuni.cz}

\begin{abstract}We establish  existence,  uniqueness and  optimal regularity results for very weak solutions to certain  nonlinear elliptic boundary value problems. We introduce structural asymptotic assumptions of Uhlenbeck type on the nonlinearity, which are sufficient and in many cases also necessary for building such a theory. We  provide a unified approach that leads qualitatively to the same theory as that one available for linear elliptic problems with continuous coefficients, e.g. the Poisson equation.

The result is based on several novel tools that are of independent interest:  local and global estimates for (non)linear elliptic systems in weighted Lebesgue spaces with Muckenhoupt weights, a generalization of the celebrated div--curl lemma for identification of a weak limit in border line spaces and the introduction of a Lipschitz approximation that is stable in weighted Sobolev spaces.
\end{abstract}

\keywords{nonlinear elliptic systems, weighted estimates, existence, uniqueness, very weak solution, monotone operator, div--curl--biting lemma, weighted space, Muckenhoupt weights}
\subjclass[2010]{35D99,35J57,35J60,35A01}

\maketitle

\section{Introduction}
\label{S1}


We study the following nonlinear problem:
for a given $n$-dimensional domain $\Omega \subset \mathbb{R}^n$ with $n\ge 2$, a~given $f:\Omega \to \mathbb{R}^{n\times N}$ with $N \in
\mathbb{N}$ arbitrary and a given mapping $A:\Omega\times
\mathbb{R}^{n\times N}\to \mathbb{R}^{n\times N}$,  to find
$u:\Omega \to \mathbb{R}^N$ satisfying
\begin{align}
\label{eq:sysA}
\begin{aligned}
  -\divergence(A(x, \nabla u)) &= -\divergence
  f &&\textrm{ in }\Omega,\\
  u&=0&&\textrm{ on }\partial \Omega.
  \end{aligned}
\end{align}

Owing to a significant number of problems originating in various applications, it is natural to require that $A$ is a Carath\'{e}odory mapping, satisfying the $p$-coercivity, the $(p-1)$ growth and the (strict) monotonicity conditions. It means that
\begin{align}
&A(\cdot,\eta) \textrm{ is measurable for any fixed } \eta\in \mathbb{R}^{n\times N},\label{Car1}\\
&A(x,\cdot) \textrm{ is continuous for almost all } x\in \Omega\label{Car2}
\end{align}
and for some $p\in (1,\infty)$, there exist positive constants $c_1$ and $c_2$ such that for almost all $x\in \Omega$ and all $\eta_1,\eta_2 \in \mathbb{R}^{n\times N}$
\begin{align}
\label{coercivityJM}
c_1|\eta_1|^p-c_2 &\le A(x,\eta_1)\cdot \eta_1 &&\textrm{$p$-coercivity},\\
|A(x,\eta_1)|&\le c_2(1+|\eta_1|)^{p-1}&&\textrm{$(p-1)$ growth},\label{growthJM}\\
0&\le (A(x,\eta_1)-A(x,\eta_2))\cdot (\eta_1-\eta_2)&&\textrm{monotonicity}.\label{monotoneJM}
\end{align}
If for all $\eta_1\neq \eta_2$ the inequality \eqref{monotoneJM} is strict, then $A$ is said to be strictly monotone.

Under the assumptions \eqref{Car1}--\eqref{monotoneJM}, it is standard to show (with the help of the Minty method, see \cite{Minty63}) that for any $f\in L^{p'}(\Omega; \mathbb{R}^{n\times N})$, with $p'$ defined as $p':=p/(p-1)$, there exists $u\in
W^{1,p}_0(\Omega;\mathbb{R}^N)$ that solves \eqref{eq:sysA} in the sense
of distribution. In addition if $A$ is strictly monotone then this solution is unique in the class of $W^{1,p}_0(\Omega;\mathbb{R}^N)$-weak solutions.

An important question that immediately arises is whether such a result can be extended to a more general setting. Namely,
\begin{equation}\tag{$\mathcal{Q}$}\label{Q}
\begin{aligned}
&\textit{whether for any $f\in L^{\frac{q}{p-1}}(\Omega;\mathbb{R}^{n\times N})$ with $q\in ((p-1),\infty)$ there exists}\\
&\textit{a (unique) $u\in W^{1,q}_0(\Omega;\mathbb{R}^N)$ solving \eqref{eq:sysA} in the weak sense.}
\end{aligned}
\end{equation}

If $q\neq p$, then we call the problem of existence and uniqueness to \eqref{eq:sysA} beyond the natural pairing. If $q>p$ and $f\in L^{q/(p-1)}(\Omega;\mathbb{R}^{n\times N})$, then $f\in L^{p'}(\Omega; \mathbb{R}^{n\times N})$ and the standard monotone operator theory in the duality pairing provides $W^{1,p}_0(\Omega;\mathbb{R}^N)$ solution to \eqref{eq:sysA}. Thus, in this case, \eqref{Q} calls only for improvement of the integrability of $\nabla u$. If $q<p$, then the considered question is apparently  more challenging as the existence of an object to start any kind of analysis with is unclear. This is  the reason why, for $(p-1)<q<p$, $W^{1,q}_0(\Omega;\mathbb{R}^N)$-solutions are called \emph{very weak solutions}.

Our general aim is to establish, for a given $f\in L^{q/(p-1)}(\Omega; \mathbb{R}^{n\times N})$ with  $q\in ((p-1),\infty)\setminus p$, the existence of a (unique) $W^{1,q}_0(\Omega;\mathbb{R}^N)$ solution to \eqref{eq:sysA}--\eqref{monotoneJM}. In this paper, we address this issue to a particular, but important case
$$
p=2.
$$

In general, the answer to \eqref{Q} is not affirmative yet even for $p=2$. This is due to two reasons:
\begin{itemize}
\item[(i)] the way how the nonlinearity $A(x,\eta)$ depends on $\eta$,
\item[(ii)] the way how the nonlinearity $A(x,\eta)$ depends on $x$.
\end{itemize}
We shall discuss each of these points from two perspectives: the available counterexamples and so far established positive results (that were rather sporadic and had several limitations).

First, we consider \eqref{eq:sysA} with $A$ depending only on $\eta$. If $q\ge p$, then there always exists  a (unique) weak solution and  the only difficult part is to obtain appropriate  a~priori estimates in the space $W^{1,q}_0(\Omega;\mathbb{R}^N)$. On the one hand, for general operators, such  a~priori estimates are not true for large $q\gg p$ even in the particular case $p=2$. This follows from the counterexamples due  Ne\v{c}as~\cite{Ne77} and \v{S}ver\'{a}k and Yan~\cite{SvYa02}, where they found  a mapping $A$, which does not depend on $x$ and satisfies\footnote{Not only the mapping $A$ satisfies \eqref{Car1}--\eqref{monotoneJM}, it has even more structure. It is given as a derivative of a uniformly convex smooth potential $F$, which makes the counterexamples even stronger.} \eqref{Car1}--\eqref{monotoneJM} with $p=2$, and showed that the corresponding unique weak solution is not in $\mathcal{C}^1$ or is even unbounded for smooth $f$.  This directly contradicts the general theory for $q\gg p$. The singular behavior of solutions in the above mentioned  counterexamples is due to the fact that the mapping $A$ depends highly nonlinearly on the \emph{vectorial} variable $\eta$. On the other hand, if $q\in [p, p+\varepsilon)$ then the $W^{1,q}_0(\Omega;\mathbb{R}^N)$ theory can be build for general mappings fulfilling only \eqref{Car1}--\eqref{monotoneJM}, where $\varepsilon>0$ depends on $c_1,c_2$. For such $q$'s, it is known that if $f\in L^{q/(p-1)}(\Omega; \setR^{n\times N})$, then there exists a solution $u\in W^{1,q}_0(\Omega;\setR^N)$ to \eqref{eq:sysA}. Such a result can be obtained by using the reverse H\"{o}lder inequality, see e.g. \cite{Gi83}. Further, based on the fundamental results of Uraltseva~\cite{Ura68} (the scalar case) and Uhlenbeck~\cite{Uhl77} (the vectorial case), for the particular choice $A(x,\eta):=|\eta|^{p-2}\eta$, or its various generalizations, the fully satisfactory theory can be built on the range $q\in [p,\infty)$, see \cite{Iwa83} for the case $A(x,\eta)=|\eta|^{p-2}\eta$ and \cite{CafPer98,DieKapSch11} for more general mappings.

For $q<p$ the situation is even more delicate. In this case, the existence of any solution is not straight forward at all. Indeed, a general existence theory for operators satisfying \eqref{Car1}--\eqref{monotoneJM} alone might be impossible to get. Up to now the only general result holds for $q\in (p-\epsilon,p+\epsilon)$ with $\epsilon$ depending only on $c_1$ and $c_2$.  It may be shown with the help of the technique  developed in \cite{Bul12} that any very weak solution to \eqref{eq:sysA} satisfies the uniform estimate
\begin{equation}\label{toco}
\int_{\Omega} |\nabla u|^q \dx \le C(c_1,c_2,q,\Omega)\int_{\Omega}|f|^{\frac{q}{p-1}}\dx \qquad \textrm{for all } q\in (p-\epsilon,p+\epsilon).
\end{equation}
In addition, if $p=2$ and  $A$ is uniformly monotone and also uniformly Lipschitz continuous, i.e.,  for all $\eta_1,\eta_2 \in \mathbb{R}^{n\times N}$ and almost all $x\in \Omega$ there holds
\begin{align}
|A(x,\eta_1)-A(x,\eta_2)|\le c_2|\eta_1-\eta_2|,\label{monotonegood}
\end{align}
then for all $f\in L^q(\Omega; \setR^{n\times N})$  there exists a unique solution $u\in W^{1,q}_0(\Omega;\setR^N)$ to \eqref{eq:sysA} whenever $q\in (2-\varepsilon, 2+\varepsilon)$, see \cite{Bul12},  and we also recall
\cite{GreIwaSbo97} for the result in the so-called grand
Lebesgue spaces $L^{(2}(\Omega)$. However, any existence theory for $q$'s ``away'' from $p$ is either missing or impossible.

More positive results are available in the scalar case $N=1$ and for the \emph{smoother} right hand side, i.e., the case when $f\in W^{1,1}(\Omega;\setR^{n})$ or at least  $f\in BV(\Omega; \setR^{n})$. Then the existence of a very weak solution is known, see the pioneering works \cite{BoGa92} and  \cite{Sta65}. Even more, one can study further qualitative properties of such a solution, see \cite{Mi13}. Moreover, in case $f\in W^{1,1}(\Omega;\setR^n)$, the uniqueness of a solution can be shown in the class of \emph{entropy} solutions, see~\cite{BoGaVa95,BoGaOr96,MasOrs97,MasOrs99}. On the other hand, in case that $f\in BV(\Omega; \setR^{n})$, or more precisely, if $\divergence f$ is only a Radon measure, the  uniqueness is not known. An exception is the case when $\divergence f$ is a finite sum of Dirac measures. In that case the study on isolated singularities by Serrin implies the uniqueness for very general non-linear operators including the $p$-Laplace equation, see \cite{Ser65,FriVer86} and references therein.  Further, this theory for measure-valued right hand side cannot be easily extend to the the case when $N>1$, where the only known result, fully compatible with the scalar case, is established for $A(x,\eta):=|\eta|^{p-2} \eta$ in \cite{DolHM97a}. To conclude this part we would like to emphasize that all results for smoother right hand side surely do not cover the full generality of the result,  we would like to have, which may be easily seen in the framework of the Sobolev embedding. Indeed, if $f\in W^{1,1}(\Omega; \setR^n)$, then $f\in L^\frac{n}{n-1}(\Omega;\setR^n)$ and we see that the case $q\in (p-1,n' (p-1))$ remains untouched even in the scalar case.

The second obstacle, related to (ii), is the possible discontinuity of the operator with respect to the spatial variable. To demonstrate this in more details, we consider the following \emph{linear} problem
\begin{align}
\label{eq:sysm}
\begin{aligned}
  -\divergence(a(x)\nabla u) &=   -\divergence f &&\textrm{ in }\Omega\\
  u&=0&&\textrm{ on }\partial \Omega
  \end{aligned}
\end{align}
with a uniformly elliptic matrix $a$. Note here, that \eqref{eq:sysm} is a particular case of \eqref{eq:sysA} with $A(x,\eta):=a(x)\eta$ and $A$ fulfilling \eqref{Car1}--\eqref{monotoneJM} with $p=2$ and $N=1$. In case $a$ is continuous and $\Omega$ is a $\mathcal{C}^1$-domain, one can
use the singular operator theory and show that for any $f\in L^{q}(\Omega;\mathbb{R}^n)$ there exists a unique weak solution $u\in
W^{1,q}_0(\Omega)$ to \eqref{eq:sysm}, see \cite[Lemma~2]{DolM95}. This can be weakened to the case when $a$ has coefficients with vanishing mean oscillations, see \cite{IwaSbo98} or \cite{DiF96}. However, the same is not true in case that $a$ is uniformly elliptic with general measurable
coefficients. Even worse, it was shown by Serrin~\cite{Ser64}, that for any $q\in (1,2)$ and $f\equiv 0$ there exists an elliptic matrix
$a$ with measurable coefficients such that one can find a distributive solution (called a \emph{pathological solution}) $v\in
W^{1,q}_0(\Omega)\setminus W^{1,2}_0(\Omega)$ that satisfies \eqref{WFG}. These pathological solutions should be excluded as only the zero function itself is the natural solution, which of course is the unique weak solution $u\in W^{1,2}_0(\Omega)$ in case $f\equiv0$. This indicates that any
reasonable theory for $q\in (1,2)$ must be able to avoid the existence of such pathological solutions.

Thus, to get a theory for all $q\in ((p-1),\infty)$, the counterexamples mentioned above indicate  that we need to assume more structural assumptions on $A$, which we shall describe in details in the next section, where we recall our problem, introduce the structural assumptions on $A$ and formulate the main results of this paper.

\section{Results}\label{S:Results}
As discussed above,  we study the problem \eqref{eq:sysA} with a mapping $A$ fulfilling \eqref{Car1}--\eqref{monotoneJM} with the particular choice $p=2$. In this case, the assumptions \eqref{Car1}--\eqref{monotoneJM} reduce to
\begin{align}
&A(\cdot,\eta) \textrm{ is measurable for any fixed } \eta\in \mathbb{R}^{n\times N},\label{2Car1}\\
&A(x,\cdot) \textrm{ is continuous for almost all } x\in \Omega\label{2Car2}
\end{align}
and to the existence of constants $c_1,c_2>0$ such that for all $\eta_1,\eta_2\in \mathbb{R}$
\begin{align}
\label{coercivity}
c_1|\eta_1|^2-c_2 &\le A(x,\eta_1)\cdot \eta_1,\\
|A(x,\eta_1)|&\le c_2(1+|\eta_1|),\label{growth}\\
0&\le (A(x,\eta_1)-A(x,\eta_2))\cdot (\eta_1-\eta_2).\label{monotone}
\end{align}



Further, inspired by the counterexamples recalled in the previous section and also by the available positive results, we shall assume in what follows that the mapping $A$ is \emph{asymptotically Uhlenbeck}, i.e.,  we will assume that there exists a continuous mapping $\tilde{A}:\overline{\Omega} \to \mathbb{R}^{n\times N}\times \mathbb{R}^{n\times N}$ fulfilling the following:
\begin{equation}\label{ass:A}
\begin{split}
&\textrm{For all $\varepsilon>0$ there exists $k>0$ such that  for almost all $x\in \Omega$ and  all $\eta\in \mathbb{R}^{n\times N}$}\\
&\textrm{satisfying $|\eta|\ge k$ there holds } \qquad \abs{A(x,\eta)-\tilde{A}(x)\eta}\le \varepsilon \abs{\eta}.
\end{split}
\end{equation}
This assumption combined with~\eqref{coercivity}--\eqref{monotone}
implies that $\tilde{A}$ necessarily satisfies
\begin{align}\label{growth2}
  c_1|\eta|^2\le \tilde{A}(x)\eta \cdot \eta \le c_2|\eta|^2 \qquad \text{for all
    $\eta \in \setR^{n \times N}$}.
\end{align}
Although the above assumption might seem to be restrictive, it enables us to cover many cases used in applications. The prototype example is of the following form
\begin{equation}\label{prominent}
A(x,\eta)=a(x,\abs{\eta})\eta \quad \textrm{with }\lim_{\lambda\to\infty}a(x,\lambda)=\tilde{a}(x), \textrm{ where } \tilde{a}\in \mathcal{C}(\overline{\Omega}).
\end{equation}
Note that $a$ may be  measurable with respect to $x$  and the required continuity must hold only for $\tilde{a}$. The assumptions \eqref{coercivity}--\eqref{monotone} are met if $a$ is strictly positive, bounded and if the function $a(x,\lambda)\lambda$ is nondecreasing with respect to $\lambda$ for almost all $x\in \Omega$. The fact that besides \eqref{2Car1}--\eqref{monotone}, we will not assume anything more but \eqref{ass:A} makes our approach general.

Moreover, to obtain the uniqueness of the solution, we will consider a stronger version of \eqref{ass:A}, namely we shall assume that $A$ is  \emph{strongly asymptotically Uhlenbeck}, i.e.,  we will assume that there exists a continuous mapping $\tilde{A}:\overline{\Omega} \to \mathbb{R}^{n\times N}\times \mathbb{R}^{n\times N}$ fulfilling the following:
\begin{equation}\label{ass:AA}
\begin{split}
&\textrm{For all $\varepsilon>0$ there exists $k>0$ such that  for almost all $x\in \Omega$  and  all $\eta\in \mathbb{R}^{n\times N}$}\\
&\textrm{satisfying $|\eta|\ge k$ there holds}\quad   \left|\frac{\partial A(x,\eta)}{\partial \eta}-\tilde{A}(x)\right|\le \varepsilon.
\end{split}
\end{equation}
Concerning the  example \eqref{prominent}, the condition \eqref{ass:AA} follows if  $a(x,\lambda)$ is differentiable with respect to $\lambda$ for $\lambda \gg 1$ and that $\lim_{\lambda\to\infty}|a'(x,\lambda)\lambda|=0$. This includes the following approximations for the $p$-Laplace operator
\[
\begin{aligned}
a(x,\abs{\eta})&=\max\set{\mu,\abs{\eta}^{p-2}} &&\textrm{for }p\in (1,2),\\
a(x,\abs{\eta})&=\min\set{\mu^{-1},\abs{\eta}^{p-2}} &&\textrm{for }p\in (2,\infty),\\
\end{aligned}
\]
which are (for small $\mu$) arbitrary close to the original setting.

\bigskip

The first main result of the paper is the following.
\begin{theorem}\label{T1}
Let $\Omega$ be a bounded $\mathcal{C}^1$-domain and $A$ satisfy
  \eqref{2Car1}--\eqref{monotone} and \eqref{ass:A}. Then for any
  $f\in L^q(\Omega; \mathbb{R}^{n\times N})$ with $q\in (1,\infty)$,
  there exists $u\in W^{1,q}_0(\Omega;\mathbb{R}^{n\times N})$ such
  that
\begin{equation}\label{WFG}
\int_{\Omega}A(x,\nabla u) \cdot \nabla \phi \dx = \int_{\Omega} f \cdot \nabla \phi \dx  \qquad \textrm{ for all } \phi\in \mathcal{C}^{0,1}_0(\Omega; \setR^N).
\end{equation}
Moreover, every very weak solution $\tilde{u}\in W^{1,\tilde{q}}_0(\Omega,\setR^N)$ to \eqref{WFG} with some $\tilde{q}>1$ satisfies
\begin{equation}\label{apriori}
\int_{\Omega}|\nabla \tilde{u}|^q \dx \le C(A,q,\Omega)\left(1+ \int_{\Omega}|f|^q\dx\right).
\end{equation}
In addition, if $A$ is strictly monotone and strongly asymptotically Uhlenbeck, i.e., \eqref{ass:AA} holds, then the solution is unique in any class $W^{1,\tilde{q}}_0(\Omega; \setR^N)$ with $\tilde{q}>1$.
\end{theorem}
Notice here, that \eqref{WFG} is nothing else than the weak formulation of \eqref{eq:sysA}. Next, we would like to emphasize the novelty of the above result.   First, to derive the estimate \eqref{apriori} one can use the comparison of \eqref{WFG} with the system with $A(x,\eta)$ replaced by  $\tilde{A}(x)\eta$ to end up with \eqref{apriori} provided that  the left hand side of \eqref{apriori} is finite a~priori. Thus, the first strong part of the result, is that \eqref{apriori} holds for {\bf all very weak solutions} to \eqref{WFG} that belong to some $W^{1,\tilde{q}}_0(\Omega;\setR^N)$ for some $\tilde{q}>1$.

Second,  Theorem~\ref{T1} implies that {\bf we can construct  solutions for the whole range $q\in (1,\infty)$}, which makes the existence theory identical with the theory for linear operators with continuous coefficients, since we know that the linear theory is not true for $q=1$ or $q=\infty$.

Third, Theorem~\ref{T1} provides {\bf the uniqueness of the very weak solution} for vector valued non-linear elliptic systems without any additional qualitative properties of a solution, e.g. the entropy inequality. In particular, the result of Theorem~\ref{T1} directly leads to the uniqueness of a solution when $\divergence f$ is a general vector-valued Radon measure. As this is of independent interest, we formulate this result in the following corollary, where we shall denote by the symbol $\mathcal{M}(\Omega;\setR^N)$ the space of $\setR^N$-valued Radon measures.
\begin{corollary}\label{CMB}
Let $\Omega$ be a bounded $\mathcal{C}^1$-domain and $A$ satisfy
  \eqref{2Car1}--\eqref{monotone} and \eqref{ass:A}. Then for any
  $f\in \mathcal{M}(\Omega; \setR^N)$ there exists $u\in W^{1,n'-\varepsilon}_0(\Omega;\mathbb{R}^{n\times N})$ with arbitrary $\varepsilon>0$ such
  that
\begin{equation}\label{WFGMB}
\int_{\Omega}A(x,\nabla u) \cdot \nabla \phi \dx = \langle f, \phi \rangle  \qquad \textrm{ for all } \phi\in \mathcal{C}^{0,1}_0(\Omega; \setR^N).
\end{equation}
Moreover, every very weak solution $\tilde{u}\in W^{1,\tilde{q}}_0(\Omega,\setR^N)$ to \eqref{WFGMB} with some $\tilde{q}>1$ satisfies for all $q\in (1,n')$ the following
\begin{equation}\label{aprioriMB}
\int_{\Omega}|\nabla \tilde{u}|^q \dx \le C(A,q,\Omega)\left(1+ \|f\|_{\mathcal{M}}^q\right).
\end{equation}
In addition, if $A$ is strictly monotone and strongly asymptotically Uhlenbeck, i.e., \eqref{ass:AA} holds, then the solution is unique in any class $W^{1,\tilde{q}}_0(\Omega; \setR^N)$ with $\tilde{q}>1$.
\end{corollary}

Although Theorem~\ref{T1} gives the final answer to \eqref{Q}, it is actually a consequence of the following stronger result. It shows the existence of a solution which is optimally smooth with respect to the right hand side in weighted spaces. For $p\in [1,\infty]$, we denote  by $\Acal_p$ the Muckenhoupt class of nonnegative weights on~$\setR^n$ (see the next section for the precise definition) and define the weighted Lebesgue space $L^p_{\omega}(\Omega):=\{f\in L^1(\Omega); \, \int_{\Omega}|f|^p\omega\dx <\infty\}$. Then, we have the following result.
\begin{theorem}\label{T3}
  Let $\Omega$ be a bounded $\mathcal{C}^1$-domain and $A$ satisfy
  \eqref{2Car1}--\eqref{monotone} and \eqref{ass:A} and let $f
  \in L^{p_0}_{\omega_0}(\Omega; \setR^{n\times N})$ for some $p_0 \in (1,\infty)$ and
  $\omega_0 \in \mathcal{A}_{p_0}$. Then there exists a $u\in
  W^{1,1}_0(\Omega;\mathbb{R}^{N})$ solving~\eqref{WFG} such that
  for all $p \in (1,\infty)$ and all weights $\omega \in \Acal_p$ the
  following estimate holds
  \begin{equation}
    \label{apriori3}
    \int_{\Omega} |\nabla u|^p\omega \dx \le C(A_p(\omega),\Omega,A,p)
    \left( 1 +  \int_{\Omega}|f|^p \omega\dx \right),
  \end{equation}
whenever the right hand side is finite.  Moreover, every very weak solution $\tilde{u}\in W^{1,\tilde{q}}_0(\Omega,\setR^N)$ to \eqref{WFG} with some $\tilde{q}>1$ satisfies \eqref{apriori3}. In addition, if $A$ is strictly monotone and strongly asymptotically Uhlenbeck, i.e.,  \eqref{ass:AA} holds, then the solution is unique in any class $W^{1,\tilde{q}}_0(\Omega; \setR^N)$ with $\tilde{q}>1$.
\end{theorem}
Clearly, Theorem~\ref{T1} is an immediate consequence of Theorem~\ref{T3}. Observe, that \eqref{apriori3} is an optimal existence result with respect to the weighted spaces. It cannot be generalized to more general weights, which is demonstrated by the theory for the Laplace equation in the whole $\setR^n$, where
one can prove that \eqref{apriori3} holds in general if and only if $\omega \in \Acal_p$. This  follows from the singular integral representation of the solution and the fundamental result of Muckenhoupt \cite{Mu72} on the continuity of the maximal function in weighted spaces.

At this point we wish to present the following corollary of Theorem~\ref{T3}. It shows, that if $f\in L^q(\Omega; \mathbb{R}^{n\times N})$ the solution constructed by Theorem~\ref{T3} implies an estimate in terms of a Hilbert space which therefore inherits the spirit of duality. Denoting by $Mf$,  the Hardy--Littlewood maximal function (see the next section for the precise definition), we have the following corollary.
\begin{corollary}\label{C2}
Let $\Omega$ be a bounded $\mathcal{C}^1$-domain and $A$ satisfy \eqref{2Car1}--\eqref{monotone} and \eqref{ass:A}. Then for any $f\in L^q(\Omega; \mathbb{R}^{n\times N})$ with $q\in (1,2]$, there exists $u\in W^{1,q}_0(\Omega;\mathbb{R}^{N})$ satisfying \eqref{WFG}. Moreover, any very weak solution $\tilde{u}\in W^{1,\tilde{q}}_0(\Omega; \setR^N)$ with some $\tilde{q}>1$ fulfilling \eqref{WFG}, satisfies the following estimate
\begin{equation}
\label{apriori2}
\int_{\Omega} \frac{|\nabla \tilde{u}|^2}{(1+Mf)^{2-q}}\dx \le C(A,q,\Omega,f) \left( 1 + \int_{\Omega}|f|^q\dx \right).
\end{equation}
\end{corollary}
As mentioned above, the estimate \eqref{apriori2} preserves the natural duality pairing in terms of weighted $L^2$ spaces, which may be of an importance in the numerical analysis. Further, as will be seen in the proof, the estimate \eqref{apriori2} plays the key role in the convergence analysis of approximate solutions to the desired one.


Next, we formulate new results that are, on the one hand essential for the proof of Theorem~\ref{T3} and  Theorem~\ref{T1}, but on the other hand of independent interest in the fields of harmonic analysis and the compensated compactness theory. These results are mainly related to two critical problems: first,  to the a~priori estimate \eqref{apriori3} and the second one, to the stability of the nonlinearity $A(x,\nabla u)$ under the weak convergence of $\nabla u$. To solve the first problem, we use the linear system as comparison to provide \eqref{apriori3}. The weighted theory for linear problems is known for $\Omega=\setR^n$ in case of constant coefficients (see e.g. \cite[p. 244]{CoiFef74}) but seems to be missing for bounded domains and linear operators continuously depending on $x$. Therefore, another essential contribution of this paper is the following theorem.
\begin{theorem}
\label{thm:ell}
Let $\Omega \subset \setR^n$ be a bounded $\mathcal{C}^1$-domain, $\omega \in \Acal_p $ for some $p\in (1,\infty)$ be arbitrary and $\tilde{A}\in \mathcal{C}(\overline{\Omega}; \setR^{n\times N\times n \times N})$ satisfy for all $z\in \setR^{n\times N}$ and all $x\in \overline{\Omega}$
\begin{equation}
c_1|\eta|^2\le \tilde{A}(x)\eta \cdot \eta \le c_2|\eta|^2
\end{equation}
with some positive constants $c_1$ and $c_2$. Then for any $f\in L^p_w(\Omega;\setR^{n\times N})$ there exists unique $v\in W^{1,1}_0(\Omega; \setR^N)$ solving
\begin{equation}
\int_{\Omega} \tilde{A}(x)\nabla v(x) \cdot \nabla \phi(x) \dx = \int_{\Omega} f(x) \cdot \nabla \phi(x) \dx \qquad \textrm{ for all } \phi\in \mathcal{C}^{0,1}_0(\Omega; \setR^N) \label{weakfppois}
\end{equation}
and fulfilling
\begin{equation}\label{keyest}
\int_\Omega \abs{\nabla v}^p\omega \dx\leq C(\Omega, \Acal_p(\omega),p,c_1,c_2)\int_\Omega \abs{f}^p\omega \dx.
\end{equation}
In addition, if  $\bar{v}\in W^{1,q}_0(\Omega; \mathbb{R}^N)$ for some $q>1$ fulfills \eqref{weakfppois} then $\bar{v}=v$.
\end{theorem}
We wish to point out, that we include natural local weighted estimates in the interior as well as on the boundary which are certainly of independent interrest (see Lemma~\ref{thm:ell-local} and Lemma~\ref{lem:halfball}).

The second obstacle we have to deal with is an identification of the weak limit and for this purpose we invent a generalization of the celebrated div--curl lemma.
\begin{theorem}[weighted, biting div--curl lemma]\label{T5}
  Let $\Omega\subset \setR^n$ be an open bounded set.
  Assume that for some $p\in (1,\infty)$ and given $\omega \in
  \Acal_p$ we have a sequence of vector-valued measurable functions $(a^k, b^k)_{k=1}^{\infty}:\Omega \to \mathbb{R}^n\times \mathbb{R}^n$ such that
  \begin{equation}
    \sup_{k\in \mathbb{N}}\int_{\Omega} \abs{a^k}^p\omega + \abs{b^k}^{p'}\omega \dx <\infty. \label{bit3}
  \end{equation}
  Furthermore, assume that for every bounded sequence
  $\{c^k\}_{k=1}^{\infty}$ from $W^{1,\infty}_0(\Omega)$ that fulfills
  $$
  \nabla c^k \rightharpoonup^* 0 \qquad \textrm{weakly$^*$ in }
  L^{\infty}(\Omega)
  $$
  there holds
  \begin{align}
    \lim_{k\to \infty} \int_{\Omega} b^k \cdot \nabla c^k \dx
    &=0, \label{bit4}
    \\
    \lim_{k\to \infty} \int_{\Omega} a^k_i \partial_{x_j} c^k -
    a^k_j \partial_{x_i} c^k \dx &=0 &&\textrm{for all }
    i,j=1,\ldots,n.\label{bit5}
  \end{align}
  Then there exists a subsequence $(a^k,b^k)$ that we do not relabel
  and there exists a non-decreasing sequence of measurable subsets
  $E_j\subset\Omega$ with $|\Omega \setminus E_j|\to 0$ as $j\to
  \infty$ such that
  \begin{align}
  a^k &\rightharpoonup a &&\textrm{weakly in } L^1(\Omega;\mathbb{R}^n), \label{bitfa}\\
  b^k &\rightharpoonup b &&\textrm{weakly in } L^1(\Omega;\mathbb{R}^n), \label{bitfb}\\
  a^k \cdot b^k \omega &\rightharpoonup a \cdot b\, \omega &&\textrm{weakly in } L^1(E_j) \quad \textrm{ for all } j\in \mathbb{N}. \label{bitf}
  \end{align}
\end{theorem}
The original version of this lemma, firstly invented by Murat and Tartar, see \cite{Mu78,Mu81,Ta78,Ta79}, was designed to identify many types of nonlinearities appearing in many types of partial differential equations. However, they assumed stronger assumptions on $a^k$ and $b^k$ than \eqref{bit4}--\eqref{bit5}, which lead to \eqref{bitf} for $E_j\equiv \Omega$.  To be more specific, they did not assume weighted spaces and considered $\omega \equiv 1$ and they required that \eqref{bit4} holds for any $c^k$ converging weakly in $W^{1,p}$ and \eqref{bit5} for any $c^k$ converging weakly in $W^{1,p'}$, respectively. The first result going more into the spirit of Theorem~\ref{T5} is due to \cite{CoDoMu11}, where the authors worked with $\omega \equiv 1$ and kept \eqref{bit4}--\eqref{bit5} but assumed the equi-integrability of the sequence $a^k\cdot b^k$. Such a result is then based on the proper use of the Lipschitz approximation of Sobolev
functions invented in \cite{AceF84}, which we shall use here as well. However, to obtain the result of Theorem~\ref{T5}, we must first use the Chacon biting lemma \cite{BroCha80,BallMurat:89}, and also improve the Lipschitz approximation method into the framework of weighted spaces, which is yet another essential result of the paper.
\begin{theorem}[Lipschitz approximation]
  \label{thm:liptrunc}
  Let $\Omega \subset \setR^n$ be an open set with Lipschitz
  boundary. Let $g \in W^{1,1}_0(\Omega; \setR^N)$. Then for
  all $\lambda >0$ there exists a \emph{Lipschitz truncation} $g^\lambda \in
  W^{1,\infty}_0(\Omega; \setR^N)$ such that
  \begin{alignat}{2}
    \label{eq:lip1}
    g^\lambda &= g \quad \text{and} \quad \nabla g^\lambda = \nabla g
    &\qquad&\text{in $\set{M(\nabla g)\leq \lambda}$ },
    \\
    \label{eq:lip2}
    \abs{\nabla g^\lambda} &\leq \abs{\nabla g}\chi_{\set{M(\nabla
        g)\leq\lambda}}+C\, \lambda
    \chi_{\set{M(\nabla g)>\lambda}} &&\textrm{almost
      everywhere}.
  \end{alignat}
  Further, if $\nabla g\in L^{p}_\omega(\Omega;\setR^{n\times N})$ for
  some $1\leq p<\infty$ and $\omega \in \mathcal{A}_p$, then
  \begin{align}\label{itm:weight}
    \begin{aligned}
      \int_{\Omega}\abs{\nabla g^\lambda}^p \omega \dx &\leq
      C(\mathcal{A}_p(\Omega),\Omega, N,p)\int_{\Omega} \abs{\nabla g}^p \omega \dx,
      \\
      \int_{\Omega} \abs{\nabla (g-g^\lambda)}^p \omega \dx&\leq
      C(\mathcal{A}_p(\Omega),\Omega, N, p) \int_{\Omega \cap \set{M(\nabla g)>\lambda}} \abs{\nabla g}^p \omega \dx.
    \end{aligned}
  \end{align}
\end{theorem}
This result has its origin in the paper \cite{AceF88}. The approach
was considerably improved and successfully used for the existence
theory in the context of fluid mechanics, see
e.g. \cite{FreMalSte00,DieMS08,DieKreSul12,Die13paseky} or \cite{BreDieFuc12,BreDieSch13} for divergence-free Lipschitz approximation. However, these results do not contain the weighted estimates~\eqref{itm:weight} and for this reason we also provide its proof in this paper.

Finally, for the sake of completeness, we present straightforward generalizations of the above results. First, we establish the theory for nonhomogeneous Dirichlet problem.
\begin{theorem}\label{T3D}
  Let $\Omega$ be a bounded $\mathcal{C}^1$-domain and $A$ satisfy
  \eqref{2Car1}--\eqref{monotone} and \eqref{ass:A} and let $f
  \in L^{p_0}_{\omega_0}(\Omega; \setR^{n\times N})$ and $u_0\in W^{1,1}(\Omega;\setR^N)$ be such that $\nabla u_0 \in L^{p_0}_{\omega_0}(\Omega;\setR^{n\times N})$ for some $p_0 \in (1,\infty)$ and
  $\omega_0 \in \mathcal{A}_{p_0}$. Then there exists a solution $u$ of~\eqref{WFG} such that $u-u_0\in
  W^{1,1}_0(\Omega;\mathbb{R}^{n\times N})$  and
  for all $p \in (1,\infty)$ and all weights $\omega \in \Acal_p$ the
  following estimate holds
  \begin{equation}
    \label{apriori3D}
    \int_{\Omega} |\nabla u|^p\omega \dx \le C(A_p(\omega),\Omega,A,p)
    \left( 1 +  \int_{\Omega}\left(|f|^p+|\nabla u_0|^p\right) \omega\dx \right),
  \end{equation}
whenever the right hand side is finite.  Moreover, every very weak solution $u$ of~\eqref{WFG} fulfilling $\tilde{u}-u_0\in W^{1,\tilde{q}}_0(\Omega,\setR^N)$  with some $\tilde{q}>1$ satisfies \eqref{apriori3D}. In addition, if $A$ is strictly monotone and strongly asymptotically Uhlenbeck, i.e.,  \eqref{ass:AA} holds, then the solution is unique in any class $W^{1,\tilde{q}}(\Omega; \setR^N)$ with $\tilde{q}>1$.
\end{theorem}
Second, we remark that for the theory for \eqref{eq:sysA} the assumptions \eqref{ass:A}--\eqref{ass:AA} are not necessary and can be weakened.
\begin{remark}
At this point we wish to discuss possible relaxations of the conditions
\eqref{ass:A} and \eqref{ass:AA} which might be useful for further
application of the theory developed here. The proof of the existence or
the uniqueness do not request, that the matrix $A(x,\eta)$ is converging
uniformly to a continuous target matrix $\tilde{A}(x)$ but rather that the
two matrices are "close" for values $\abs{\eta}>k$ for some $k$. Indeed,
it is possible to quantify the necessary closeness in accordance to the
ellipticity and continuity parameters of $\tilde{A}(x)$ and
$\partial\Omega$.
A different relaxation of \eqref{ass:A} and \eqref{ass:AA} could be done
in a non-pointwise manner: By replacing the pointwise asymptotic
conditions by asymptotic conditions in terms of vanishing mean
oscillations (VMO).
\end{remark}

We conclude this section by  highlighting the essential novelties of this paper:
\begin{itemize}
\item[1)] A complete unified $W^{1,q}_0(\Omega;\mathbb{R}^N)$-theory for nonlinear elliptic systems with the asymptotic Uhlenbeck structure satisfying \eqref{2Car1}--\eqref{monotone}, \eqref{ass:A} and \eqref{ass:AA} has been was developed in such a way that the theory is identical with that one for linear operators with continuous coefficients -- Theorem~\ref{T1} and Theorem~\ref{T3D}. Moreover, the new estimate suitable  for numerical purposes is established in Corollary~\ref{C2}.
\item[2)] A maximal regularity in weighted spaces of any very weak solution is established as well as its uniqueness, which in particular leads to the uniqueness of very weak solution to the problems with measure right hand side -- Theorem~\ref{T3} for the nonlinear case and Theorem~\ref{thm:ell} for the linear setting.
\item[3)] A new tool in harmonic analysis, the Lipschitz approximation method in weighted spaces, is developed -- Theorem~\ref{thm:liptrunc}.
\item[4)] A new tool for identification of a weak limit of the nonlinear operator, the biting weighted div--curl lemma, is invented -- Theorem~\ref{T5}. Such tool has a potential  to improve  the known methods in compensated compactness theory in significant manner.
\end{itemize}
To summarize, this paper proposes a new way how to attack more general elliptic problems than those discussed in Section~\ref{S:Results}. Indeed, it seems that the only missing point in the analysis of more general problems, e.g., the $p$-Laplace equation, is the \emph{formal} a~priori estimates beyond the duality pairing. Once such a~priori estimates are available one can follow the method introduced in this paper and gain an existence and uniqueness theory for general problems beyond the natural duality.

\bigskip

The structure of the remaining part of the paper is the following. In Section~\ref{S2}, we specify auxiliary tools used in the proofs of the results, Sections~\ref{S3}--\ref{Last:S} are devoted to the proofs of the main theorems of the paper and Section~\ref{P:C} is dedicated to the proofs of the corollaries.

\section{Auxiliary tools}
\label{S2}


\subsection{Muckenhoupt weights and the maximal function}
We start this part by recalling the definition of  the Hardy Littlewood maximal function. For any $f\in L^1_{\loc}(\setR^n)$ we define
\[
Mf(x):=\sup_{R>0}\dashint_{B_R(x)}\abs{f(y)}\dy \quad \textrm{with} \quad \dashint_{B_R(x)}\abs{f(y)}\dy:=\frac{1}{|B_R(x)|}\int_{B_R(x)}\abs{f(y)}\dy,
\]
where $B_R(x)$ denotes a ball with radius $R$ centered at $x\in \setR^n$. We shall use  similar notation for the vector- or tensor-valued function as well. Note here, that we could replace balls in the definition of the maximal function by cubes with sides parallel to the axis without any change. We will also use in what follows the standard notion for Lebesgue and Sobolev spaces. Further,  we say that $\omega:\setR^n \to \setR$ is a weight if it is a measurable function that is almost everywhere finite and positive. For such a weight and arbitrary measurable $\Omega\subset \setR^n$ we denote the space $L^p_{\omega}(\Omega)$ with $p\in [1,\infty)$ as
$$
L^p_{\omega}(\Omega):=\biggset{u:\Omega \to \setR^n; \; \norm{f}_{L^p_\omega}
:= \bigg(\int_{\Omega} |u(x)|^p\omega(x)\dx\bigg)^{\frac 1p} <\infty}.
$$
Note that our weights are defined on the whole space~$\setR^n$. Next, for  $p\in [1,\infty)$, we say that a weight $\omega$
belongs to the Muckenhoupt class $\Acal_p$ if and only if
there exists a positive constant $A$ such that for every balls $B
\subset \setR^n$ the following holds
\begin{alignat}{2}
  \label{defAp2}
  \left(\dashint_B
    \omega\dx\right)\left(\dashint_B\omega^{-(p'-1)}\dx\right)^\frac{1}{p'-1}
  &\le A & \qquad\qquad&\text{if $p \in (1,\infty)$},
  \\
  \label{defA1}
  M\omega(x)&\le A\, \omega(x) &&\text{if $p=1$}.
\end{alignat}
In what follows, we denote by
$A_p(\omega)$ the smallest constant $A$ for which the
inequality~\eqref{defAp2}, resp.~\eqref{defA1}, holds. Due to the celebrated result of Muckenhoupt, see \cite{Mu72}, we know
that $\omega \in \mathcal{A}_p$ is for $1<p<\infty$ equivalent to the
existence of a constant $A'$, such that  for all $f\in L^p(\setR^n)$
\begin{equation}\label{defAp}
\int\abs{Mf}^p\omega\dx\leq A'\,\int\abs{f}^p\omega\dx.
\end{equation}
Further, if $p \in [1,\infty)$ and $\omega \in \mathcal{A}_p$, then we have an embedding  $L^p_\omega (\Omega)
\embedding L^1_{\loc}(\Omega)$, since for all balls~$B \subset \setR^n$ there holds
\begin{align*}
  \dashint_B \abs{f}\dx &\leq \bigg(\dashint_B \abs{f}^p
  \omega\dx\bigg)^{\frac 1p} \bigg(\dashint_B
  \omega^{-(p'-1)}\dx\bigg)^{\frac 1{p'}} \leq \big(
  \mathcal{A}_p(\omega)\big)^{\frac 1p} \bigg( \frac{1}{\omega(B)}
  \int_B \abs{f}^p \omega\dx\bigg)^{\frac 1p}.
\end{align*}
In particular, the distributional derivatives of all $f \in
L^p_\omega$ are well defined. Next, we summarize some properties of Muckenhoupt weights in the following lemma.
\begin{lemma}[Lemma 1.2.12 in \cite{Tur00}]
  \label{cor:leftopen}
  Let $\omega\in \Acal_p$ for some $p\in [1,\infty)$. Then $\omega\in
  \Acal_q$ for all $q\ge p$. Moreover, there
  exists $s=s(p,A_p(\omega))>1$ such that $\omega \in
  L^s_{\loc}(\setR^n)$ and we have the reverse H\"{o}lder inequality,
  i.e.,
  \begin{equation}
    \left(\dashint_B \omega^s \dx\right)^{\frac{1}{s}}\le
    C(n,A_p(\omega))\dashint_B \omega \dx. \label{reversom}
  \end{equation}
  Further, if $p \in (1,\infty)$, then there exists
 $\sigma=\sigma(p,A_p(\omega)) \in [1,p)$ such that
  $\omega\in \mathcal{A}_\sigma$. In addition,   $\omega \in \mathcal{A}_p$ is equivalent to $\omega^{-(p'-1)} \in
  \mathcal{A}_{p'}$.
\end{lemma}
In the paper, we also use the following improved embedding $L^p_\omega (\Omega)
\embedding L^q_{\loc}(\Omega)$ valid for all $\omega \in \mathcal{A}_p$ with $p \in
(1,\infty)$ and some $q\in [1,p)$ depending only on $A_p(\omega)$. Such an embedding can be deduced by a direct application of Lemma~\ref{cor:leftopen}. Indeed, since $\omega \in \mathcal{A}_p$, we have $\omega^{-(p'-1)} \in
\mathcal{A}_{p'}$. Thus, using Lemma~\ref{cor:leftopen}, there exists $s=s(A_p(\omega))>1$ such that
\begin{align*}
  \left(\dashint_B \omega^{-s(p'-1)} \dx\right)^{\frac{1}{s}}\le
  C(A_p(\omega))\dashint_B \omega^{-(p'-1)} \dx.
\end{align*}
Consequently, for $q:= \frac{sp}{p+s-1} \in (1,p)$ we can use the H\"{o}lder inequality to deduce that
\begin{align}
  \label{eq:lqprop}
  \begin{aligned}
    \bigg( \dashint_B \abs{f}^q\dx\bigg)^{\frac 1q} &\leq
    \bigg(\dashint_B \abs{f}^p \omega\dx\bigg)^{\frac 1p}
    \bigg(\dashint_B \omega^{-s(p'-1)}\dx\bigg)^{\frac 1{s p'}}\\
&\leq C(A_p(\omega)) \bigg( \frac{1}{\omega(B)} \int_B
    \abs{f}^p \omega\dx\bigg)^{\frac 1p},
  \end{aligned}
\end{align}
which implies the desired embedding.

The next result makes another link between the maximal function and  $\Acal_p$-weight.
\begin{lemma}[See pages 229--230 in \cite{86:_real} and page 5 in \cite{Tur00}]\label{cor:dual}
Let $f \in L^1_{\loc}(\setR^n)$ be such that $Mf<\infty$ almost everywhere in $\setR^n$. Then for all $\alpha \in (0,1)$ we have $(Mf)^{\alpha} \in \Acal_1$. Furthermore, for all $p\in (1,\infty)$ and all $\alpha\in (0,1)$ there holds
$(Mf)^{-\alpha(p-1)}\in \Acal_p$.
\end{lemma}

We would also like  to point out that the maximum $\omega_1 \vee \omega_2$
and minimum $\omega_1 \wedge \omega_2$ of two $\Acal_p$-weights are again
$\Acal_p$-weights. For $p=2$ we even have $A_2(\omega_1 \wedge
\omega_2) \leq A(\omega_1) +A_2(\omega_2)$, which follows from the following simple computation
\begin{align}
  \label{eq:A2min}
  \begin{aligned}
    \dashint_B \omega_1 \wedge \omega_2\dx \dashint_B \frac1{\omega_1
      \wedge \omega_2}\dx &\leq \left[\bigg(\dashint_B \omega_1\dx\bigg)
    \wedge \bigg( \dashint_B\omega_2\dx \bigg)\right] \dashint_B
    \frac{1}{\omega_1} + \frac{1}{\omega_2}\dx
    \\
    &\leq A_2(\omega_1) + A_2(\omega_2).
  \end{aligned}
\end{align}

\subsection{Convergence tools}
The results recalled in the previous sections shall give us a direct way for a~priori estimates for an approximative problem \eqref{eq:sysA}. However, to identify the limit correctly, we use Theorem~\ref{T5}, which is based on the following  biting lemma.
\begin{lemma}[Chacon's Biting Lemma, see \cite{BallMurat:89}]\label{thm:blem}
   Let $\Omega$ be a bounded domain in $\setR^n$ and let
   $\{v^n\}_{n=1}^{\infty}$ be a bounded sequence in $L^1(\Omega)$. Then
   there exists a non-decreasing  sequence of measurable subsets
   $E_j\subset\Omega$ with $|\Omega \setminus E_j|\to 0$ as $j\to \infty$ such that $\{v^n\}_{n\in\mathbb{N}}$ is
   pre-compact in the weak topology of $L^1(E_j)$, for each $j\in\mathbb{N}$.
 \end{lemma}
Note here, that pre-compactness of $v^n$ is equivalent to the following:  for every $j\in \mathbb{N}$ and every $\varepsilon>0$ there exists a $\delta>0$ such that for all $A\subset E_j$ with $\abs{A}\leq \delta$ and all $n\in \mathbb{N}$ the following holds
\begin{equation}\label{smallunif}
\int_A\abs{v^n}\dx\leq \epsilon.
\end{equation}

\subsection{$L^q$-theory for linear systems with continuous coefficients}
The starting point for getting all a~priori estimates in the paper is the following.
\begin{lemma}[See Lemma~2 in \cite{DolM95}]
  \label{lem:CZ}
  Let $\Omega$ be a $\mathcal{C}^1$-domain and $\bfB \in
  \mathcal{C}(\overline{\Omega}, \setR^{n\times N \times n\times N})$ be a
  continuous, elliptic tensor that satisfies for all $\eta\in
  \setR^{n\times N}$ and all $x\in \overline{\Omega}$
  \begin{equation}
    c_1|\eta|^2\le \bfB(x)\eta \cdot \eta \le c_2|\eta|^2
  \end{equation}
  for some $c_1,c_2>0$.  Then for any $f\in L^q(\Omega; \setR^{n\times
    N})$ with $q\in (1,\infty)$ there exists unique $w\in
  W^{1,q}_0(\Omega; \setR^N)$ solving
  \begin{align*}
    -\divergence (\bfB\nabla w)&=-\divergence f \qquad \text{ in } \Omega
  \end{align*}
  in the sense of distribution.
  Moreover, there exists a constant $C$ depending only on $\bfB$, $q$
  and the shape of $\Omega$ such that
  \begin{equation}\label{aoprioriconst}
    \norm{\nabla w}_{L^{q}(\Omega)}\leq C(\bfB,q,\Omega)\norm{f}_{L^{q}(\Omega)}.
  \end{equation}
\end{lemma}

\section{Proof of Theorems~\ref{T1} and \ref{T3}}
\label{S3}
%
%
%
First, it is evident that Theorem~\ref{T1} directly follows from Theorem~\ref{T3} by setting $\omega \equiv 1$, which is surely an $\mathcal{A}_p$-weight. Therefore, we focus on the proof of Theorem~\ref{T3}. We split the proof into several steps. We start with the uniform estimates, which heavily relies on Theorem~\ref{thm:ell}, then provide the existence proof, for which we use the result of Theorem~\ref{T5}, and finally show the uniqueness of the solution, again based on Theorem~\ref{thm:ell}.
\subsection{Uniform estimates}
We start the proof by showing the uniform estimate \eqref{apriori3} for arbitrary $u\in W^{1,q}_0(\Omega; \mathbb{R}^N)$ with $q>1$ solving \eqref{WFG}. Without loss of generality, we can restrict ourselves to the case $q\in (1,2)$. First, we consider the case when $f\in L^{2}_{\omega}(\Omega;\mathbb{R}^{n\times N})$ with some weight $\omega \in \mathcal{A}_2$. Then for fix $j \in \mathbb{N}$ we define an auxiliary weight $\omega_j:=\omega \wedge j(1+M|\nabla u|)^{q-2}$. Then it follows from Lemma~\ref{cor:dual} and the fact that $q\in (1,2)$ that $w_j\in \mathcal{A}_2$. Moreover, we have
$$
A_2(\omega_j)\le A_2(\omega) + A_2(j(1+M|\nabla u|)^{q-2})= A_2(\omega) + A_2((1+M|\nabla u|)^{q-2})\le C(u,\omega)
$$
and also that $\nabla u, f \in L^2_{\omega_j}(\Omega; \setR^{n\times N})$. Next, using \eqref{WFG}, we see that for all  $\phi \in \mathcal{C}_0^{0,1}(\Omega; \mathbb{R}^N)$
\begin{align}
  \int_{\Omega} \tilde{A}(x)\nabla u\cdot \nabla \phi\dx =
  \int_{\Omega}(f -A(x,\nabla u)+\tilde{A}(x)\nabla u)\cdot
  \nabla \phi\dx.\label{startqw}
\end{align}
Since the right hand side belongs to
$L^2_{\omega_j}(\Omega; \mathbb{R}^{n\times N})$, we can use Theorem~\ref{thm:ell} and the assumptions \eqref{growth} and \eqref{growth2} to get the
estimate
\begin{align*}
  \int_\Omega \abs{\nabla u}^2\omega_j \dx &\leq
  C(\tilde{A},A_2(\omega_j),\Omega,c_1,c_2)\int_\Omega\abs{f-A(x,\nabla
    u)+\tilde{A}(x)\nabla u}^2\omega_j \dx
  \\
  &\leq C(\tilde{A},u,\omega,\Omega,c_1,c_2)\left(\int_\Omega\abs{f}^2
    \omega_j \dx +\int_{\Omega}\abs{A(x,\nabla u)-\tilde{A}(x)\nabla
      u}^2\omega_j \dx \right)
  \\
  &\leq C(\tilde{A},u,\omega,\Omega,c_1,c_2)\int_\Omega(\abs{f}^2+k^2)\, \omega_j
  \dx
  \\
  &\qquad +C(\tilde{A},u,\omega,\Omega,c_1,c_2) \int_{\{|\nabla u|\ge
    k\}}\frac{\abs{A(x,\nabla u)-\tilde{A}(x)\nabla u}^2}{|\nabla
    u|^2} |\nabla u|^2\omega_j \dx.
\end{align*}
Finally, we set
$$
\varepsilon^2:=\frac{1}{2C(\tilde{A},u,\omega,\Omega,c_1,c_2)}
$$
and according to \eqref{ass:A} we can find $k$ such that
$$
\frac{\abs{A(x,\nabla u)-\tilde{A}(x)\nabla u}^2}{|\nabla u|^2}\le \frac{1}{2C(\tilde{A},u,\omega,\Omega,c_1,c_2)},
$$
provided that $|\nabla u|\ge k$. Inserting this inequality above, we deduce that
\begin{align*}
  \int_\Omega \abs{\nabla u}^2\omega_j \dx &\leq
  C(\tilde{A},u,\omega,\Omega,c_1,c_2)\int_\Omega(\abs{f}^2+k^2) \omega_j \dx
  +\frac12\int_{\Omega}|\nabla u|^2\omega_j \dx.
\end{align*}
Since we already know that $\nabla u \in L^2_{\omega_j}(\Omega;\setR^{n\times N})$ and $k$ is fixed independently of $j$,  we
can absorb the last term into the left hand side to get
\begin{align*}
  \int_\Omega \abs{\nabla u}^2\omega_j \dx &\leq
  C(\tilde{A},u,\omega,\Omega,c_1,c_2)\int_\Omega(\abs{f}^2+1) \omega_j \dx.
\end{align*}
Next, we let $j\to \infty$ in the above inequality. For the right hand side, we use the fact that $\omega_j \leq \omega$ and for the left hand side we use the
monotone convergence theorem (notice here that $\omega_j\nearrow \omega$ since $M|\nabla u| < \infty$ almost everywhere) for the left hand side to obtain
\begin{align*}
  \int_\Omega \abs{\nabla u}^2\omega \dx &\leq
  C(\tilde{A},u,\omega,\Omega,c_1,c_2)\left(1+\int_\Omega\abs{f}^2\omega
    \dx\right).
\end{align*}
Although this estimate is not uniform  yet, since the right hand side still depends on the $A_2$ constant of $(1+M|\nabla u|)^{q-2}$, it  implies that $\nabla u\in L^2_{\omega}(\Omega;\setR^{n\times N})$ for the original weight $\omega$. Therefore, we can reiterate this procedure, i.e., going back to \eqref{startqw} and applying Theorem~\ref{thm:ell}, we find that
\begin{align*}
  \int_\Omega \abs{\nabla u}^2\omega \dx &\leq
  C(\tilde{A},A_2(\omega),\Omega,c_1,c_2)\int_\Omega\abs{f-A(x,\nabla
    u)+\tilde{A}(x)\nabla u}^2\omega \dx
  \\
  &\leq C(\tilde{A},A_2(\omega),\Omega,c_1,c_2)\int_\Omega(\abs{f}^2+k)\, \omega
  \dx
  \\
  &\qquad +C(\tilde{A},A_2(\omega),\Omega,c_1,c_2) \int_{\{|\nabla u|\ge
    k\}}\frac{\abs{A(x,\nabla u)-\tilde{A}(x)\nabla u}^2}{|\nabla
    u|^2} |\nabla u|^2\omega \dx.
\end{align*}
Since we already know that  $\nabla u  \in L^2_{\omega}(\Omega;\setR^{n\times N})$, we
can use the same procedure as above and absorb the last term into the left hand side to get
\begin{align}
  \label{eq:final6}
  \int_\Omega \abs{\nabla u}^2\omega \dx &\leq
  C(c_1,c_2,A_2(\omega),\Omega,\tilde{A})\left(1+\int_\Omega\abs{f}^2\omega
    \dx\right).
\end{align}
We would like to emphasize  that the constant $C$ in~\eqref{eq:final6}  depends on~$\omega$ only through its $A_2$-constant. Therefore, by the {\em miracle of extrapolation} \cite[Theorem 3.1]{CruMP06} (see also~\cite{Rub84})
applied to the couples $(\nabla u, f)$ we can extend this estimate
valid for all $\mathcal{A}_2$-weights to all~$\mathcal{A}_p$-weights. In particular, we find that
\begin{align*}
  \int_\Omega \abs{\nabla u}^p\omega \dx &\leq
  C(c_1,c_2,A_p(\omega),\Omega,\tilde{A})\left(1+\int_\Omega\abs{f}^p\omega
    \dx\right) \qquad \text{for all $1<p<\infty$ and $\omega \in
    \mathcal{A}_p$},
\end{align*}
which is just~\eqref{apriori3} from our claim.


\subsection{Existence of a solution}
Let $f\in L^p_{\omega}(\Omega;\mathbb{R}^{n\times N})$ with some $p\in (1,\infty)$ and $\omega \in \mathcal{A}_p$ be arbitrary. Then according to~\eqref{eq:lqprop} there exists some $q_0 \in (1,2)$ such that
$L^{p}_{\omega}(\Omega) \embedding L^{q_0}(\Omega)$. Therefore, defining $\omega_0 := (1+ Mf)^{q_0-2}$, we can use Lemma~\ref{cor:dual} to obtain that $\omega_0 \in \mathcal{A}_2$ and it is evident that  $f \in L^2_{\omega_0}(\Omega; \setR^{n\times N})$.

The construction of the solution is based on a proper approximation of
the right hand side $f$ and a limiting procedure. We first extend $f$
outside of $\Omega$ by zero and define $f^k:=f \chi_{\set{|f|<k}}$.
Then $f^k$ are bounded functions, $|f^k|\nearrow |f|$ and
\begin{align}\label{cfn}
  f^k \to f &&\textrm{strongly in } L^2_{\omega_0}\cap
  L^{q_0}(\setR^n; \setR^{n\times N}).
\end{align}
For such an approximative $f^k$ we can use the standard monotone operator theory to find a solution $u^k\in W^{1,2}_0(\Omega; \mathbb{R}^{N})$ fulfilling
\begin{equation}
  \int_{\Omega} A(x,\nabla u^k) \cdot \nabla \phi \dx = \int_{\Omega}f^k
  \cdot \nabla \phi \dx   \qquad \textrm{ for all } \phi\in W^{1,2}_0(\Omega; \mathbb{R}^N). \label{wfn}
\end{equation}
Hence, we can use the already proven estimate \eqref{apriori3} to deduce that
\begin{equation}
\begin{aligned}
  \label{finaln}
  \int_\Omega \abs{\nabla u^k}^2\omega_0 \dx &\leq
  C(c_1,c_2,A_2(\omega_0),\Omega,\tilde{A})\left(1+\int_\Omega\abs{f^k}^2\omega_0
    \dx\right)\\
    &\leq
  C(c_1,c_2,q_0,f,A_2(\omega_0),\tilde{A})\left(1+\int_\Omega\abs{f}^2\omega_0
    \dx\right)\\
    &\le C(c_1,c_2,\Omega,\tilde{A},f,\omega).
\end{aligned}
\end{equation}
Using the estimate~\eqref{finaln}, the reflexivity of the corresponding spaces, the embedding $L^2_{\omega_0}(\Omega) \embedding
L^{q_0}(\Omega)$ and the growth assumption~\eqref{growth},  we can pass to a subsequence (still denoted
by $u^k$) such that
\begin{align}
  \label{conn-a}
  u^k &\rightharpoonup u &&\textrm{weakly in } W^{1,q_0}_0(\Omega; \setR^N),
  \\
  \label{conn-b}
  \nabla u^k &\rightharpoonup \nabla u &&\textrm{weakly in }
  L^2_{\omega_0}\cap L^{q_0}(\Omega; \setR^{n\times N}),
  \\
  A(x,\nabla u^k) &\rightharpoonup \overline{A} &&\textrm{weakly in }
  L^2_{\omega_0}\cap L^{q_0}(\Omega; \setR^{n\times N})\label{con2}.
\end{align}
Next, using \eqref{finaln}--\eqref{conn-b}, the weak lower semicontinuity
and the unique identification of the limit~$u$ in $W^{1,1}(\Omega)$, we obtain
\begin{align}
  \label{eq:final}
  \int_\Omega \abs{\nabla u}^2\omega_0 \dx &\leq
  C(c_1,c_2,A_2(\omega_0),\Omega,\tilde{A})\left(1+\int_\Omega\abs{f}^2\omega_0
    \dx\right).
\end{align}

The last step is to show that~$u$ is a solution to our problem, i.e., that it satisfies \eqref{WFG}.  Using
\eqref{wfn}, \eqref{cfn} and \eqref{con2} it follows that
\begin{align}
  \label{eq:limit}
  \int_\Omega \Adash\cdot\nabla \phi \dx=\int_\Omega f\cdot\nabla \phi
  \dx\qquad\text{ for all }\phi\in \mathcal{C}^{0,1}_0(\Omega;\mathbb{R}^N).
\end{align}
Hence, to complete the existence part of the proof of Theorem~\ref{T3}, it remains to show that
\begin{align}
  \overline{A}(x)&=A(x,\nabla u(x)) \qquad\textrm{in  } \Omega.\label{show22}
\end{align}
To do so, we use\footnote{Although Theorem~\ref{T5} is formulated for
   vector-valued functions, it is an easy extension to use it also
  for  matrix-valued functions, which is the case here.}
Theorem~\ref{T5}. We denote $a^k:=\nabla u^k$ and $b^k:=A(x,\nabla
u^k)$. By using \eqref{finaln} and
\eqref{growth}, we find that \eqref{bit3} is satisfied with the weight $\omega_0$. Also the assumption \eqref{bit4} holds, which follows from  \eqref{cfn},
\eqref{wfn} and \eqref{eq:limit}. Finally, \eqref{bit5} is valid
trivially since $a^k$ is a gradient. Therefore, Theorem~\ref{T5} can
be applied, which implies the existence of a non-decreasing sequence of measurable
sets $E_j$, such that $|\Omega \setminus E_j|\to 0$ and
\begin{equation}\label{Minty}
  A(x,\nabla u^k) \cdot \nabla u^k \omega_0 \rightharpoonup \overline{A}
  \cdot  \nabla u\, \omega_0 \qquad \textrm{weakly in } L^1(E_j).
\end{equation}
For any $B\in L^{2}_{\omega_0}(\Omega; \setR^{n\times N})$, we have  that $B\, \omega_0$ and also $A(\cdot,B)\,
\omega_0$ belong to  $L^2_{1/\omega_0}(\Omega;\setR^{n\times N})$ and therefore using~\eqref{conn-b} and~\eqref{con2},  we can observe that
\begin{equation}\label{Minty2}
  (A(x,\nabla u^k)-A(x,B)) \cdot (\nabla u^k-B) \,\omega_0 \rightharpoonup
  (\overline{A}-A(x,B)) \cdot (\nabla u-B)\, \omega_0 \quad
  \textrm{weakly in }  L^1(E_j).
\end{equation}
Due to the monotonicity of $A$ we see that the term on the left hand side is non-negative and consequently, its weak limit is non-negative as well and we have that
\begin{equation}\label{Minty3}
  \int_{E_j}(\overline{A}-A(x,B)) \cdot (\nabla u-B) \,\omega_0 \dx \ge 0
  \qquad \textrm{for all } B\in L^2_{\omega_0}(\Omega;\setR^{n\times N})
  \textrm{ and all }  j\in \mathbb{N}.
\end{equation}
Therefore, it follows that
\begin{equation*}
\int_{\Omega}(\overline{A}-A(x,B)) \cdot (\nabla u-B)\, \omega_0 \dx \ge \int_{\Omega\setminus E_j}(\overline{A}-A(x,B)) \cdot (\nabla u-B)\, \omega_0 \dx
\end{equation*}
and letting $j \to \infty$ (note that the integral is well defined due
to~\eqref{conn-b} and~\eqref{con2}),  using the fact that $|\Omega
\setminus E_j|\to 0$ as $j\to \infty$ and the Lebesgue dominated convergence theorem, we obtain
\begin{equation*}
\int_{\Omega}(\overline{A}-A(x,B)) \cdot (\nabla u-B)\, \omega_0 \dx \ge 0 \qquad \textrm{for all } B\in L^2_{\omega_0}(\Omega;\setR^{n\times N}).
\end{equation*}
Hence, setting $B:=\nabla u -\varepsilon G$ where $G\in L^{\infty}(\Omega; \setR^{n\times N})$ is arbitrary and dividing by $\varepsilon$ we get
\begin{equation*}
  \int_{\Omega}(\overline{A}-A(x,\nabla u - \varepsilon G)) \cdot G\, \omega_0 \dx \ge 0 \qquad \textrm{for all } G\in L^{\infty}(\Omega;\setR^{n\times N}).
\end{equation*}
Finally, using the Lebesgue dominated convergence theorem, the assumption \eqref{growth} and the continuity of $A$ with respect to the second variable, we can let $\varepsilon \to 0_+$ to deduce
\begin{equation*}
\int_{\Omega}(\overline{A}-A(x,\nabla u)) \cdot G\, \omega_0 \dx \ge 0 \qquad \textrm{for all } G\in L^{\infty}(\Omega;\setR^{n\times N}).
\end{equation*}
Since $\omega_0$ is strictly positive almost everywhere in $\Omega$, the relation \eqref{show22} easily follows by setting e.g.,
$$
G:=-\frac{\overline{A}-A(x,\nabla u)}{1+\abs{\overline{A}-A(x,\nabla u)}}.
$$
Thus~\eqref{eq:limit} follows and $u$ is a very weak solution.

\subsection{Uniqueness} Let $u_1,u_2 \in W^{1,q}_0(\Omega;\setR^N)$ with $q>1$ be two very weak solutions to  \eqref{WFG} for some given $f\in L^p_{\omega}(\Omega;\setR^{n\times N})$, where $p\in (1,\infty)$ and $\omega \in \Acal_p$. Then it directly follows that
\begin{equation}
\label{un1}
\int_{\Omega} (A(x,\nabla u_1)-A(x,\nabla u_2))\cdot \nabla \varphi \dx = 0 \qquad \textrm{ for all } \varphi\in \mathcal{C}^{0,1}_0(\Omega;\setR^{n\times N}).
\end{equation}
First, consider the case that $f\in L^2(\Omega; \setR^{n\times N})$. Then using the result of the previous part we see that $u_1,u_2 \in W^{1,2}_0(\Omega;\setR^N)$ and therefore due to the growth assumption \eqref{growth}, we see that \eqref{un1} is valid for all $\varphi\in W^{1,2}_0(\Omega; \setR^N)$. Consequently, the choice $\varphi:=u_1-u_2$ is admissible and due to the strict monotonicity of $A$ we conclude that $\nabla u_1=\nabla u_2$ almost everywhere in $\Omega$ and due to the zero trace also that $u_1=u_2$.

Thus, it remains to discuss the case $f\notin L^2(\Omega;\setR^{n\times N})$. But since $f\in L^p_{\omega}(\Omega;\setR^{n\times N})$ with $p>1$ and $\omega$ being the $\mathcal{A}_p$-weight, we can deduce that $f\in L^{p_0}(\Omega;\setR^{n\times N})$ for some $p_0>1$, see \eqref{eq:lqprop}. Consequently, following Lemma~\ref{cor:dual}, we can define the $\mathcal{A}_2$-weight $\omega_0:=(1+Mf)^{p_0-2}$ and we get that $f\in L^2_{\omega_0}(\Omega;\setR^{n\times N})$. Therefore, using the maximal regularity result we can deduce that $\nabla u_i\in L^2_{\omega_0}(\Omega;\setR^{n\times N})$ for $i=1,2$. Hence, defining a new weight $w^n:=1\wedge (n\omega_0)$, which is bounded,  we also get that for each $n$ the solutions satisfy $\nabla u_i \in L^2_{\omega^n}(\Omega;\setR^{n\times N})$. Moreover, we have the estimate $A_2(\omega^n)\le A_2(1)+A_2(n\omega_0)=1+A_{2}(\omega_0)\le C(f)$. Hence, rewriting the identity \eqref{un1} into the form
\begin{equation}
\label{un2}
\begin{split}
&\int_{\Omega}\tilde{A}(x)(\nabla u_1-\nabla u_2) \cdot \nabla \varphi \dx \\
&\quad = \int_{\Omega} \left(\tilde{A}(x)\nabla u_1 - A(x,\nabla u_1) -(\tilde{A}(x)\nabla u_2-A(x,\nabla u_2))\right)\cdot \nabla \varphi \dx,
\end{split}
\end{equation}
which is valid for all $\varphi\in \mathcal{C}^{0,1}_0(\Omega;\setR^{n\times N})$, we can use Theorem~\ref{thm:ell} to obtain
\begin{equation}
\label{un3}
\int_{\Omega}|\nabla u_1-\nabla u_2|^2 \omega^n\dx \le C \int_{\Omega} \left|\tilde{A}(x)\nabla u_1 - A(x,\nabla u_1) -(\tilde{A}(x)\nabla u_2-A(x,\nabla u_2))\right|^2 \omega^n \dx
\end{equation}
with some constant $C$ independent of $n$. Moreover due to the properties of the solution and $\omega^n$ we can deduce that the integral appearing on the right hand side is finite. In order to continue, we first recall the following algebraic result, whose proof can be found at the end of this subsection.
\begin{lemma}\label{L:algebra}
Let $A$ fulfill \eqref{coercivity}, \eqref{growth}, \eqref{ass:A} and \eqref{ass:AA}. Then for every  $\delta>0$ there exists $C$ such that for all $x\in \Omega$ and all $\eta_1, \eta_2 \in \mathbb{R}^{n\times N}$ there holds
\begin{equation}\label{algebra}
|A(x,\eta_1)-A(x,\eta_2) -\tilde{A}(x)(\eta_1-\eta_2)|\le \delta |\eta_1-\eta_2| + C(\delta).
\end{equation}
\end{lemma}
Next, using the estimate \eqref{algebra} in \eqref{un3}, we find that for all $\delta>0$
\begin{equation}
\label{un4}
\begin{split}
\int_{\Omega}|\nabla u_1-\nabla u_2|^2 \omega^n\dx \le C \int_{\Omega} \delta|\nabla u_1-\nabla u_2|^2\omega^n + C(\delta)\omega^n\dx.
\end{split}
\end{equation}
Thus, setting $\delta:=\frac{1}{2C}$, we can deduce that
\begin{equation}
\label{un5}
\begin{split}
\int_{\Omega}|\nabla u_1-\nabla u_2|^2 \omega^n\dx \le C(\delta) \int_{\Omega} \omega^n\dx\le C,
\end{split}
\end{equation}
where the last inequality follows from the fact that $\Omega$ is bounded and $\omega^n \le 1$. Hence, letting $n\to \infty$ in \eqref{un5}, using  that $\omega^n \nearrow 1$ (which follows from the fact that $\omega_0>0$  almost everywhere) and using the monotone convergence theorem, we find that
$$
\int_{\Omega}|\nabla u_1-\nabla u_2|^2\dx \le C.
$$
Hence, we see that  $u_1-u_2 \in W^{1,2}_0(\Omega; \setR^N)$. In addition, using \eqref{algebra} again, we have that
\begin{equation*}
\begin{split}
&\int_{\Omega}|A(x,\nabla u_1)-A(x,\nabla u_2)|^2\dx \\
&\le 2\int_{\Omega}|A(x,\nabla u_1)-\tilde{A}(x)\nabla u_1 -A(x,\nabla u_2)+\tilde{A}(x)\nabla u_2|^2\dx \\
&\quad +2\int_{\Omega}|\tilde{A}(x)\nabla u_1 -\tilde{A}(x)\nabla u_2|^2\dx \le C\left(1+\int_{\Omega} |\nabla u_1 -\nabla u_2|^2\dx\right) \le C.
\end{split}
\end{equation*}
Therefore, \eqref{un1} holds for all $\varphi\in W^{1,2}_0(\Omega; \setR^{n\times N})$ and consequently also for $\varphi:=u_1-u_2$ and the strict monotonicity finishes the proof of the uniqueness. It remains to prove Lemma~\ref{L:algebra}.
\begin{proof}[Proof of Lemma~\ref{L:algebra}] Let $\delta$ be given and fix. According to \eqref{ass:A} and \eqref{ass:AA} we can find $k>0$ (depending on $\delta$) such that for all $x\in \Omega$ and all $|\eta|\ge k$ we have
\begin{equation}\label{ass:F}
\frac{|A(x,\eta)-\tilde{A}(x)\eta|}{|\eta|}+\left|\frac{\partial A(x,\eta)}{\partial \eta} - \tilde{A}(x)\right| \le \frac{\delta}{4}.
\end{equation}
To prove \eqref{algebra} we shall discus all possible cases of values $\eta_1, \eta_2$. Recall here, that $\delta$ and $k$ are already fix.

\bigskip
\paragraph{\bf The case $|\eta_1|\le 2k$ and $|\eta_2|\le 2k$}In this case we can simply use \eqref{growth} to show that
$$
|A(x,\eta_1)-A(x,\eta_2) -\tilde{A}(x)(\eta_1-\eta_2)|\le C(1+|\eta_1|+|\eta_2|)\le C(1+4k)
$$
and \eqref{algebra} follows.

\bigskip
\paragraph{\bf The case $|\eta_1|\le 2k$ and $|\eta_2|> 2k$} In this case we again use \eqref{growth} and combined with \eqref{ass:F} leads  to
$$
\begin{aligned}
&|A(x,\eta_1)-A(x,\eta_2) -\tilde{A}(x)(\eta_1-\eta_2)|\le C(1+|\eta_1|) + \left|\frac{\tilde{A}(x)\eta_2 - A(x,\eta_2)}{|\eta_2|}\right||\eta_2|\\
&\le C(1+2k)+\frac{\delta |\eta_2|}{2} \le C(1+2k+|\eta_1|)+\frac{\delta |\eta_2-\eta_1|}{2}\le C(1+4k)+\delta |\eta_2-\eta_1|.
\end{aligned}
$$
Therefore, \eqref{algebra} holds. Moreover, the case $|\eta_1|\ge 2k$ and $|\eta_2|\le 2k$ is treated similarly.

\bigskip
\paragraph{\bf The case $|\eta_1|> 2k$ and $|\eta_2|> 2k$}
First, let us also assume that
\begin{equation}\label{alas}
|\eta_2|\le 2|\eta_1-\eta_2|\qquad \textrm{ and } \qquad |\eta_1|\le 2|\eta_1-\eta_2|
\end{equation}
 In this setting, we use \eqref{ass:F} to conclude
$$
\begin{aligned}
&|A(x,\eta_1)-A(x,\eta_2) -\tilde{A}(x)(\eta_1-\eta_2)|\\
&\qquad \le \left|\frac{\tilde{A}(x)\eta_1 - A(x,\eta_1)}{|\eta_1|}\right||\eta_1| + \left|\frac{\tilde{A}(x)\eta_2 - A(x,\eta_2)}{|\eta_2|}\right||\eta_2| \le \frac{\delta}{4} (|\eta_1|+|\eta_2|)\le \delta|\eta_1-\eta_2|,
\end{aligned}
$$
which again directly implies \eqref{algebra}. Finally, it remains to discuss the case when at least one of the inequalities in \eqref{alas} does not hold. For simplicity, we consider only the case when $|\eta_1|> 2|\eta_1-\eta_2|$, since the second case can be treated similarly.  First of all, using the assumption on $\eta_1$ and $\eta_2$ we deduce that for all $t\in [0,1]$
$$
|t\eta_2 + (1-t)\eta_1|=|\eta_1 -t(\eta_1-\eta_2)|\ge |\eta_1|-t|\eta_1-\eta_2|\ge |\eta_1|-|\eta_1-\eta_2|\ge \frac{|\eta_1|}{2} \ge k.
$$
Hence, since any convex combination of $\eta_1$ and $\eta_2$ is outside of the ball or radius $k$, we can use the assumption \eqref{ass:F} to conclude
$$
\begin{aligned}
&|A(x,\eta_2)-A(x,\eta_1) -\tilde{A}(x)(\eta_2-\eta_1)|\\
&\quad =\left|\int_0^1\frac{d}{dt} \left(A(x,t\eta_2+(1-t)\eta_1) - \tilde{A}(x)(t\eta_2+(1-t)\eta_1)\right) \dt \right|\\
&\quad =\left|\int_0^1 \left(\frac{\partial A(x,t\eta_2+(1-t)\eta_1)}{\partial (t\eta_2+(1-t)\eta_1)} - \tilde{A}(x)\right)(\eta_2-\eta_1) \dt \right| \le\int_0^1 \frac{\delta}{4}|\eta_2-\eta_1| \dt \le\delta |\eta_2-\eta_1|
\end{aligned}
$$
and \eqref{algebra} follows.
\end{proof}

\section{Proof of Theorem~\ref{thm:ell}}\label{S:forjistota}
We start the proof by getting the a~priori estimate in the standard non-weighted Lebesgue  spaces, which is available due to Lemma~\ref{lem:CZ}.
Let us fix a ball~$Q_0$ such that $\Omega \subset Q_0$.  Since $\omega
\in \mathcal{A}_p$, we can use~\eqref{eq:lqprop} to show that for some
$\tilde{q}>1$, we have $L^p_\omega(Q_0) \embedding
L^{\tilde{q}}(Q_0)$. Thus $f \in L^p_\omega(\Omega; \setR^{n\times N})$ implies that $f \in
L^{\tilde{q}}(\Omega; \setR^{n\times N})$. The starting point of  further analysis is the use of Lemma~\ref{lem:CZ}, which leads to  the existence of a unique solution $u\in W^{1,\tilde{q}}_0(\Omega;\setR^{N})$ to \eqref{weakfppois} with the a~priori bound
\begin{align*}
  \bigg(\int_{\Omega}|\nabla u|^{\tilde{q}}\dx \bigg)^{\frac
    1{\tilde{q}}} \le C(A,\tilde{q},\Omega)\bigg(
  \int_{\Omega}|f|^{\tilde{q}}\dx \bigg)^{\frac 1{\tilde{q}}}.
\end{align*}
Consequently, using~\eqref{eq:lqprop}, we deduce
\begin{align}
  \bigg(\frac 1{\abs{Q_0}} \int_{\Omega}|\nabla u|^{\tilde{q}}\dx
  \bigg)^{\frac 1{\tilde{q}}} \leq C(A,p,\Omega,
  \mathcal{A}_p(\omega)) \bigg( \frac{1}{\omega(Q_0)}
  \int_{\Omega}|f|^p \omega\,dx \bigg)^{\frac 1p}.
  \label{stth-new}
\end{align}
It remains to prove the a~priori estimate~\eqref{keyest}.  We divide
the proof into several steps.  In the first one, we shall prove the
local (in $\Omega$) estimates. Then we extend such a result up to the
boundary and finally we combine them together to get
Theorem~\ref{thm:ell}.  \bigskip
\subsection{Interior estimates:} This part is devoted to the estimates that are local in $\Omega$, i.e., we shall prove the following.
\begin{lemma}
\label{thm:ell-local}
Let $B\subset \setR^n$ be a ball, $\omega \in \Acal_p$ be arbitrary with some $p\in (1,\infty)$ and $A\in L^{\infty}(2B;\setR^{n\times N \times n \times N})$ be arbitrary satisfying
$$
c_1 |\eta|^2\le A(x) \eta \cdot \eta \le c_2|\eta|^2 \textrm{ for all } x\in 2B \textrm{ and all } \eta \in \setR^{n\times N}.
$$
Then there exists $\delta>0$ depending only on $p$, $c_1$, $c_2$ and $A_p(\omega)$ such that if
$$
|A(x)-A(y)| \le \delta \textrm{ for all } x,y\in 2B
$$
then for arbitrary $f\in L^p_\omega(2B; \setR^{n\times N})$ and $u\in W^{1,\tilde{q}}(2B;\setR^N)$ with some $\tilde{q}>1$ satisfying
$$
\int_{2B} A(x) \nabla u(x) \cdot \nabla \varphi(x) \dx= \int_{2B} f(x) \cdot \nabla \varphi(x) \dx\quad \textrm{ for all } \varphi \in \mathcal{C}^{0,1}_0(2B; \setR^N),
$$
the following holds
\begin{equation}\label{unlocal}
  \bigg(\dashint_B\abs{\nabla u}^p\omega \dx\bigg)^\frac1p\leq
  C\bigg(  \dashint_{2B}\abs{f}^p\omega \dx\bigg)^\frac1p
  + C \bigg(\dashint_{2B}\omega \dx\bigg)^{\frac{1}{p}}\bigg(\dashint_{2B}\abs{\nabla u}^{\tilde{q}} \dx\bigg)^{\frac{1}{\tilde{q}}},
\end{equation}
where the constant $C$ depends only  on $p$, $c_1$, $c_2$ and $A_p(\omega)$.
\end{lemma}
\begin{proof}[Proof of Lemma~\ref{thm:ell-local}]
First, we introduce some more notation. For $\omega$ we denote $\omega(S):=\int_S \omega \dx$. Next, using Lemma~\ref{cor:leftopen}, we can find $q\in (1,\tilde{q})$ such that $\omega \in \Acal_{\frac{p}{q}}$. Note here that $u\in W^{1,q}(2B;\setR^{N})$, which follows from the fact that $2B$ is bounded. In what follows, we fix such $q$ and introduce the centered maximal operator with power $q$
\[
 (M_q(g))(x):=\sup_{r>0}\bigg(\dashint_{B_r(x)}\abs{g}^q \dy\bigg)^\frac1q.
\]
Since, $M_q(g)=(M(\abs{g}^q))^\frac1q$, we see that from the definition and the choice of $q$ (which leads to $\omega\in \Acal_\frac{p}{q}(\setR^n)$) that  the operator $M_q$ is bounded in $L^p_{\omega}(\setR^n)$. We shall also use the restricted maximal operator
\[
 (M_q^{<\rho}(g))(x)=\sup_{\rho\geq r>0}\bigg(\dashint_{B_r(x)}\abs{g}^q dy\bigg)^\frac1q
\]
and it directly follows that for every Lebesgue point $x$ of $g$
\[
 |g(x)|\leq (M_q^{<\rho}(g))(x)\leq (M_q(g))(x).
\]

The inequality \eqref{unlocal} will be proven be using the proper estimates on the  level sets for $|\nabla u|$ defined through
\[
 O_\lambda:=\set{x\in \setR^n; \, M_q(\chi_{2B}\nabla u)(x)> \lambda}.
\]
Please observe that $O_{\lambda}$ are open. Next, we use the \Calderon-Zygmund decomposition. Thus, for fixed $\lambda>0$ and $x\in B\cap Q_{\lambda}$, using the continuity  of the integral with respect to the integration domain, we can find a ball $Q_{r_x}(x)$, such that
\begin{align}
\label{eq:cover1}
 \lambda^q< \dashint_{Q_{r_x}(x)}\abs{\chi_{2B}\nabla u}^q dx\leq 2\lambda^q\quad \text{ and } \quad\dashint_{Q_{r}(x)}\abs{\chi_{2B}\nabla u}^q dx \leq 2\lambda^q\text{ for all }r\geq r_{x}.
\end{align}
Next, using the Besicovich covering theorem, we can extract a countable
subset $Q_i:=Q_{r_i}(x_i)$, such that the
$Q_i$ have finite intersection, i.e., there exists a constant $C$ depending only on $n$ such that for all $i\in \mathbb{N}$
$$
\# \{j\in \mathbb{N}; \, Q_i\cap Q_j \neq \emptyset\} \le C.
$$
In addition, it follows from the construction that
\begin{align}\label{realn}
O_{\lambda}\cap B=\bigcup_{i\in \mathbb{N}} (Q_i\cap B).
\end{align}
Then we set
\[
  \Lambda:=\bigg(\dashint_{2B} \abs{\nabla u}^q dx\bigg)^\frac1q
\]
and it directly follows that for any $Q\subset \setR^n$
\[
 \bigg(\dashint_Q \abs{\chi_{2B}\nabla u}^q dx\bigg)^\frac1q\leq \left(\frac{\abs{2B}}{\abs{Q}}\right)^\frac1q\Lambda.
\]
Consequently, assuming that $\lambda\geq 2^{2n} \Lambda$, we can deduce for every $Q_i$ that
\[
 2^{2n} \Lambda \le \lambda < \bigg(\dashint_{Q_i} \abs{\chi_{2B}\nabla u}^q dx\bigg)^\frac1q \le \left(\frac{\abs{2B}}{\abs{Q_i}}\right)^\frac1q\Lambda= 2^{\frac{2n}{q}} \left(\frac{\abs{B}}{\abs{2Q_i}}\right)^\frac1q\Lambda.
\]
Since $q\ge 1$, this inequality directly leads to $\abs{2 Q_i}\leq \abs{B}$. Therefore, using the fact that $Q_i=Q_{r_i}(x_i)$ with some $x_i\in B$, we observe that  $2Q_i\subset 2B$. Moreover, it is evident that for some constant $C$ depending only on the dimension $n$, we have
\begin{align}
\label{eq:cover2}
\abs{Q_i}\le C(n)\abs{Q_i\cap B}.
\end{align}
Since $\omega\in \mathcal{A}_p$, the above relation implies
(see e.g. V~1.7~\cite{Ste93})
\begin{align}
  \label{eq:cover2omega}
  \omega(Q_i)\le C(n, A_p(\omega))\, \omega(Q_i\cap B).
\end{align}
Next, for arbitrary $\epsilon>0$ and $k\ge 1$, we introduce the re-distributional set
\[
 U^\lambda_{\epsilon, k}:=O_{k\lambda}\cap\set{x\in \setR^n;\, M_q(f\chi_{2B})(x)\leq \epsilon\lambda}.
\]
Finally, we shall assume the  following (recall that $\delta$ comes from the assumption of Lemma~\ref{thm:ell-local}):
\begin{equation}
\begin{aligned}
 \label{eq:redis}
 &\textrm{There exists $k\ge 1$ depending only on $c_1$, $c_2$, $n$, $p$, $A_p(\omega)$ such that for all $\varepsilon\in  (0,1)$}\\
  &\textrm{and all $\lambda\ge 2^{2^n}\Lambda$ there holds }\qquad
\abs{Q_i\cap U^\lambda_{\epsilon, k}\cap B}\leq C(c_1,c_2,n)(\epsilon +\delta)\abs{Q_i}.
\end{aligned}
\end{equation}

We postpone the proof of \eqref{eq:redis}  and continue assuming that it holds true with fix $k$ such  that \eqref{eq:redis} is valid. Hence, using \eqref{eq:redis},  the  H\"older and the reverse H\"older inequality (which follows
for $\Acal_p$-weights from \eqref{reversom})
and~\eqref{eq:cover2omega}, we obtain for some $r>1$ depending only on
$n$, $p$ and $A_p(\omega)$
\begin{align*}
  &\omega(Q_i \cap U_{\epsilon, k}^\lambda \cap B) \leq C(n) \abs{Q_i}
  \bigg(\dashint_{Q_i}\omega^r
  \dx\bigg)^\frac1r\left(\frac{\abs{Q_i\cap
        U^\lambda_{\epsilon,k}\cap B}}{\abs{Q_i}}\right)^\frac1{r'}
  \\
  &\quad \leq C(n,p,A_p(\omega),c_1,c_2)(\epsilon+\delta)^\frac1{r'}
  \omega(Q_i)\leq  C(n,p,A_p(\omega),c_1,c_2)(\epsilon+\delta)^\frac1{r'}
  \omega(Q_i \cap B).
\end{align*}
By using the finite intersection property of the~$Q_i$ we find
\begin{align}
\label{eq:smallness}
\omega(U_{\epsilon, k}^\lambda \cap B)\leq C(n,A_p(\omega),c_1,c_2)(\epsilon+\delta)^\frac1{r'}\omega(O_\lambda\cap B).
\end{align}
Finally, using the Fubini theorem, we obtain
\begin{align}\label{sce}
 \int_B\abs{\nabla u}^p \omega \dx = p\int_0^\infty \!\!\omega(\set{(\nabla u)\chi_{B}>\lambda})\lambda^{p-1} {\rm{d}}\lambda\leq \Lambda^p\omega(B) + p\int_{\Lambda}^\infty\lambda^{p-1}\omega(O_\lambda \cap B)\rm{d}\lambda.
\end{align}
Therefore, to get the estimate \eqref{unlocal}, we need to estimate the last term on the right hand side. To do so, we use the definition of $U^{\lambda}_{\epsilon, k}$ and the substitution theorem, which leads for all $m> k\Lambda$ to
\begin{align*}
 \int_{k\Lambda}^m\lambda^{p-1}\omega(O_\lambda \cap B)\rm{d}\lambda& \leq \int_{k\Lambda}^m\lambda^{p-1}\omega(U^{\frac{\lambda}{k}}_{\epsilon, k}\cap B){\rm{d}}\lambda
+ \int_{k\Lambda}^m\lambda^{p-1}\omega\Big(\bigset{M_q(f\chi_{2B})>\epsilon\frac{\lambda}{k}}\Big){\rm{d}}\lambda\\
&\overset{\eqref{eq:smallness}}\leq C(\epsilon+\delta)^\frac1{r'}\int_{k\Lambda}^m\lambda^{p-1}\omega(O_\frac{\lambda}{k}\cap B){\rm{d}}\lambda+ \frac{k^p}{p\epsilon^p}\int_{\setR^n}\abs{M_q(f\chi_{2B})}^p \omega \dx
\end{align*}
\begin{align*}
&\le C(p,q,\epsilon,A_p(\omega))\int_{2B}\abs{f}^p \omega \dx +Ck^{p}(\epsilon+\delta)^\frac1{r'}\int_{\Lambda}^\frac{m}{k}\lambda^{p-1}\omega(O_\lambda\cap B){\rm{d}}\lambda\\
&\le C(p,q,\epsilon,A_p(\omega))\int_{2B}\abs{f}^p \omega \dx +Ck^{p}(\epsilon+\delta)^\frac1{r'}\int_{\Lambda}^{k\Lambda}\lambda^{p-1}\omega(O_\lambda\cap B){\rm{d}}\lambda\\
&\quad +
Ck^{p}(\epsilon+\delta)^\frac1{r'}\int_{k\Lambda}^m\lambda^{p-1}\omega(O_\lambda\cap B){\rm{d}}\lambda,
\end{align*}
where  we used the fact that $\omega \in \Acal_{\frac{p}{q}}$. Finally, assuming (note that $k$ is already fix by \eqref{eq:redis} and at this point, we fix the maximal value of $\delta$ arising in the assumption of Lemma~\ref{thm:ell-local}) that $\delta$ is so small that $Ck^{p}\delta\frac1{r'}\le\frac18$, we can find  $\epsilon\in (0,1)$ such that $Ck^{p}(\epsilon+\delta)^\frac1{r'}\le\frac12$. Consequently, we  absorb the last term into the left hand side and letting $m\to \infty$, we find that
\begin{align*}
 \int_{k\Lambda}^{\infty}\lambda^{p-1}\omega(O_\lambda \cap B)\rm{d}\lambda
&\le C(k,p,q,A_p(\omega))\left(\int_{2B}\abs{f}^p \omega \dx+ \Lambda^p \omega(B)\right).
\end{align*}
Substituting this into \eqref{sce}, we find \eqref{unlocal}. To finish the proof, it remains to find $k\ge 1$ such that  \eqref{eq:redis} holds.

Hence, assume that $Q_i\cap B\cap U^\lambda_{\epsilon, k}\neq\emptyset$. Then it follows from the definition of $U^{\lambda}_{\epsilon,k}$ that
\begin{align}\label{smalfe}
 \bigg(\dashint_{2Q_i}\abs{f}^q \dx\bigg)^\frac1q\leq 2^n\epsilon \lambda.
\end{align}
For $\lambda \geq 2^{2n} \Lambda$ (which implies $2Q_i \subset 2B$) we
compare the original problem with the following
\begin{align}
 \label{eq:har}
\begin{aligned}
-\divergence (A_i\nabla h) &= 0 &&\text{ in }2Q_i\\
h&=u &&\text{ on }\partial (2Q_i),
\end{aligned}
\end{align}
where the matrix $A_i$ is defined as $A_i:=A(x_i)$. Lemma~\ref{lem:CZ}
ensures the existence of such a solution (just consider $u-h$ with
zero boundary data). Moreover, the matrix $A_i$ is constant and elliptic and therefore we have the weak Harnack
inequality for $h$, i.e.,
\begin{align}
  \label{eq:harnack}
  \sup_{\frac32 Q_i}\abs{\nabla h}\leq
  C\dashint_{2Q_i}\abs{\nabla h}\dx,
\end{align}
where the constant $C$ depends only on $n$, $c_1$ and
$c_2$. Further, since $u$ solves our original problem, we find
$$
\begin{aligned}
-\divergence(A_i\nabla(u-h))&=-\divergence((A-A_i)\nabla u-f) &&\textrm{in }2Q_i,\\
u-h&=0 &&\textrm{on }\partial 2Q_i.
\end{aligned}
$$
Therefore, we can use Lemma~\ref{lem:CZ} to observe
\begin{align}
\label{eq:comp}
 \dashint_{2Q_i}\abs{\nabla(u-h)}^q\, \dx\leq C\dashint_{2Q_i}\abs{A-A_i}^q\abs{\nabla u}^q \dx + C\dashint_{2Q_i}\abs{f}^q \dx\leq C(\epsilon^q +\delta^q)\lambda^q,
\end{align}
where for the second inequality we used \eqref{eq:cover1},
\eqref{smalfe} and the assumption that $|A(x)-A(y)|\le \delta$ for all $x,y\in B$. Then using the definition of $Q_i$, we see that for
all  $y\in Q_i$ and all $r>\frac{r_i}{2}$, we have that $B_{r}(y)\subset B_{3r}(x_i) $ and $Q_i\subset B_{3r}(x_i)$. Consequently,
\[
 \dashint_{B_r(y)}\abs{\chi_{2B}\nabla u}^q \dx \leq 3^n \dashint_{B_{3r}(x_i)}\abs{\chi_{2B}\nabla u}^q \dx  \leq 6^n\lambda^q,
\]
where we used \eqref{eq:cover1}. Choosing $k\geq 6^n$ and assuming
that $\epsilon,\delta \le 1$ we get by the previous estimate, the
sub-linearity of the maximal operator and the weak  Harnack inequality \eqref{eq:harnack}  that for all $x\in Q_i\cap \set{M_q(\nabla u)>k\lambda}$
\begin{align*}
 M_q(\nabla u)(x)&=M_q^{<\frac{r_i}{2}}(\nabla u)(x)
\leq M_q^{<\frac{r_i}{2}}(\nabla h)(x)+M_q^{<\frac{r_i}{2}}(\nabla u-\nabla h)(x)\\
 &\leq  C\bigg(\dashint_{2Q_i}\abs{\nabla h}^q \dx\bigg)^\frac1q +M_q^{<\frac{r_i}{2}}(\nabla u-\nabla h)(x)\leq C\lambda+M_q^{<\frac{r_i}{2}}(\nabla u-\nabla h)(x).
\end{align*}
Hence, setting $k:=\max\set{C+1,6^{n}}$, we can use the weak $L^q$ estimate for the maximal functions and the estimate \eqref{eq:comp} to conclude
\begin{align*}
 \abs{\set{M_q(\nabla u)> k\lambda}\cap Q_i}&\leq \abs{ \set{M_q^{<\frac{r_i}{2}}(\nabla u-\nabla h)\geq \lambda}\cap Q_i}\leq \frac{C}{\lambda^q} \int_{2Q_i}\abs{\nabla (u-h)}^q \dx\\
 &\leq C(\epsilon+\delta)\abs{Q_i},
\end{align*}
which finishes the proof of \eqref{eq:redis} and Lemma~\ref{thm:ell-local}.
\end{proof}

\subsection{Estimates near the boundary:} In this part, we generalize the result from the previous paragraph and extend its validity also to the neighborhood of the boundary.
\begin{lemma}
\label{lem:halfball}
Let $\Omega\subset \setR^n$ be a domain with $\mathcal{C}^1$ boundary,  $\omega \in \Acal_p $ be arbitrary with some $p\in (1,\infty)$ and $A\in L^{\infty}(\Omega;\setR^{n\times N \times n \times N})$ be arbitrary satisfying
$$
c_1 |\eta|^2\le A(x) \eta \cdot \eta \le c_2|\eta|^2 \textrm{ for all } x\in 2B \textrm{ and all } \eta \in \setR^{n\times N}.
$$
Then there exists $r^*>0$ and $\delta>0$ depending only on $\Omega$, $p$, $c_1$, $c_2$ and $A_p(\omega)$ such that if
$$
\sup_{x,y \in \Omega; \, |x-y|\le r^*}|A(x)-A(y)| \le \delta
$$
then for arbitrary $f\in L^p_\omega(\Omega; \setR^{n\times N})$ and $u\in W^{1,\tilde{q}}_0(\Omega;\setR^N)$ with some $\tilde{q}>1$ satisfying
\begin{equation}\label{wfL2}
\int_{\Omega} A \nabla u \cdot \nabla \varphi \dx= \int_{\Omega} f \cdot \nabla \varphi \dx\quad \textrm{ for all } \varphi \in \mathcal{C}^{0,1}_0(\Omega; \setR^N),
\end{equation}
we have  for all $x_0 \in \overline{\Omega}$ and all $r\le r^*$ the following estimate
\begin{equation}\label{uptobound}
\begin{split}
\dashint_{B_{r}(x_0)\cap\Omega}\!\!\!\abs{\nabla u}^p\omega
\dx &\leq
\dashint_{B_{2r}(x_0)\cap\Omega} \!\!\!C\abs{f}^p\omega
\dx+ \dashint_{B_{2r}(x_0)\cap\Omega}\!\!\!\omega
\dx
\bigg(\dashint_{B_{2r}(x_0)\cap\Omega} \!\!\!C\abs{\nabla u}^{\tilde{q}} \dx\bigg)^{\frac{p}{\tilde{q}}}.
\end{split}
\end{equation}
\end{lemma}

First notice that in case $B_{2r}(x_0)\subset \Omega$ the inequality \eqref{uptobound} follows from Lemma~\ref{thm:ell-local}. Therefore, we focus only on the behavior near the boundary. Hence, let $x_0\in \partial \Omega$ be arbitrary. Since $\Omega\in \mathcal{C}^1$, we know that there exists $\alpha, \beta>0$ and $r_0>0$ such that (after a possible change of coordinates)
$$
\begin{aligned}
B_{r_0}^+&:=\{(x',x_n):\, |x'|<\alpha, \; a(x')-\beta< x_n <a(x')\}\subset \Omega,\\
B_{r_0}^{-}&:=\{(x',x_n):\, |x'|<\alpha, \; a(x')< x_n <a(x')+\beta\}\subset \Omega^c.
\end{aligned}
$$
Here, we abbreviated $(x_1,\ldots, x_n):= (x',x_n)$. Moreover, we know that for all $r\le \frac{r_0}{2}$ it holds $B_{2r}(x_0)\cap \Omega \subset B_{r_0}^+$ and $B_{2r}(x_0)\cap \Omega^c \subset B_{r_0}^-$. In addition, we have $a\in \mathcal{C}^1([-\alpha,\alpha]^{n-1})$ and $\nabla a(0)\equiv 0$. For later purposes we also denote
$$
B_{r_0}:=B_{r_0}^+\cup B_{r_0}^{-} \cup \{(x,x_n); \, |x'|<\alpha, \, a(x')=x_n\}
$$
and  define a mapping $T:B_{r_0}^+ \to B_{r_0}^-$ as
$$
T(x',x_n):=(x',2a(x')-x_n) \qquad \textrm{ with } \qquad J(x):=\nabla T(x), \textrm{ i.e., } (J(x))_{ij}:=\partial_{x_j} (T(x))_i.
$$
It directly follows from the definition that $|\det J(x)|\equiv 1$ and also that $T$ and $T^{-1}$ are $\mathcal{C}^1$ mappings. Finally, we extend all quantities into $B_{r_0}^-$ as follows:
\begin{align*}
\tilde{u}(x)&:=\left\{\begin{aligned}
&u(x) &&\textrm{for } x\in B_{r_0}^+,\\
&-u(T^{-1}(x)) &&\textrm{for } x\in B_{r_0}^{-},
\end{aligned}
\right.  \\
\tilde{A}(x)&:=\left\{\begin{aligned}
&A(x) &&\textrm{for } x\in B_{r_0}^+,\\
&J(T^{-1}x)A(T^{-1}x)J^T(T^{-1}x) &&\textrm{for } x\in B_{r_0}^{-},
\end{aligned}
\right.  \\
\tilde{f}(x)&:=\left\{\begin{aligned}
&f(x) &&\textrm{for } x\in B_{r_0}^+,\\
&-J(T^{-1}x)f(T^{-1}(x)) &&\textrm{for } x\in B_{r_0}^{-},
\end{aligned}
\right.  \\
\tilde{\omega}(x)&:=\left\{\begin{aligned}
&\omega(x) &&\textrm{for } x\in B_{r_0}^+,\\
&\omega(T^{-1}(x)) &&\textrm{for } x\in B_{r_0}^{-}.
\end{aligned}
\right.
\end{align*}
It also directly follows from the definition and the fact that $u$ has zero trace on $\partial \Omega$ that $\tilde{u} \in W^{1,q}(B_{r_0}; \setR^N)$.
Finally, we show that for all $\varphi \in \mathcal{C}^{0,1}_0(B_{r_0};\setR^N)$ the following identity holds
\begin{equation}
\int_{B_{r_0}} \tilde{A} \nabla \tilde{u} \cdot \nabla \varphi \dx = \int_{B_{r_0}} \tilde{f} \cdot \nabla \varphi \dx.\label{extend}
\end{equation}
For this we observe that for any $\phi \in \mathcal{C}^{0,1}_0(B^-_{r_0};
\setR^N)$ and $\hat{\phi} := \phi \circ T \in \mathcal{C}^{0,1}_0(B^+_{r_0};
\setR^N)$ we have
\begin{align*}
  \int_{B^{-}_{r_0}} &(\tilde{A}\nabla \tilde{u} - \tilde{f}) \cdot
  \nabla \varphi \dx =\int_{B^{-}_{r_0}} \left(\tilde{A}^{\mu
      \nu}_{ij}(x) \frac{\partial \tilde{u}^{\nu}(x)}{\partial x_j} -
    \tilde{f}^{\mu}_i(x)\right)
  \frac{\partial\varphi^{\mu}(x)}{\partial x_i} \dx
  \\
  &=\int_{B^{-}_{r_0}} \left(-\tilde{A}^{\mu \nu}_{ij}(x)
    \frac{\partial (u^{\nu}(T^{-1}x))}{\partial x_j} -
    \tilde{f}^{\mu}_i(x)\right)
  \frac{\partial(\hat{\varphi}^{\mu}(T^{-1}(x)))}{\partial x_i} \dx
  \\
  &=\int_{B^{-}_{r_0}} \left(-\tilde{A}^{\mu \nu}_{ij}(x)
    \frac{\partial u^{\nu}(T^{-1}x)}{\partial (T^{-1}(x))_k}
    J^{-1}_{kj}(T^{-1}(x)) - \tilde{f}^{\mu}_i(x)\right)
  \frac{\partial\hat{\varphi}^{\mu}(T^{-1}(x))}{\partial
    (T^{-1}(x))_m} J^{-1}_{mi}(T^{-1}(x)) \dx
\\
  &=\int_{B^{+}_{r_0}} \left(-\tilde{A}^{\mu \nu}_{ij}(Tx)
    \frac{\partial u^{\nu}(x)}{\partial x_k}
    J^{-1}_{kj}(x)J^{-1}_{mi}(x) -
    \tilde{f}^{\mu}_i(Tx)J^{-1}_{mi}(x)\right)
  \frac{\partial\hat{\varphi}^{\mu}(x)}{\partial x_m} \dx
  \\
  &=-\int_{B^{+}_{r_0}} \left(A(x) \nabla u(x) - f(x)\right) \cdot
  \nabla \hat{\varphi}(x)\dx.
\end{align*}
In particular, for all $\phi \in \mathcal{C}^{0,1}_0(B^+_{r_0}; \setR^N)$ we have
\begin{align}
  \label{eq:mirror}
  \int_{B^{-}_{r_0}} (\tilde{A}\nabla \tilde{u} - \tilde{f}) \cdot
   \nabla (\varphi \circ T^{-1}) \dx &=-\int_{B^{+}_{r_0}} \left(A
     \nabla u - f\right) \cdot \nabla \phi\dx.
\end{align}
Thus, if we define for $\phi \in \mathcal{C}^{0,1}_0(B_{r_0};\setR^N)$
the function
\begin{align*}
  \overline{\phi} &:=
  \begin{cases}
    \phi \circ T^{-1} &\qquad \text{on } B^-_{r_0},
    \\
    \phi &\qquad \text{on } B^+_{r_0},
  \end{cases}
\end{align*}
then $\overline{\phi} \in \mathcal{C}^{0,1}_0(B_{r_0};\setR^N)$
and~\eqref{eq:mirror} implies
\begin{align*}
  \int_{B_{r_0}} \big(\tilde{A} \nabla \tilde{u} -
    \tilde{f}\big) \cdot \nabla \overline{\varphi}\dx &= 0.
\end{align*}
Therefore,
\begin{align*}
  \int_{B_{r_0}} \big(\tilde{A} \nabla \tilde{u} - \tilde{f}\big)
  \cdot \nabla \varphi\dx &= \int_{B_{r_0}} \big(\tilde{A} \nabla
  \tilde{u} - \tilde{f}\big) \cdot \nabla (\varphi -
  \overline{\phi})\dx
  =\int_{B_{r_0}^-} \big(\tilde{A} \nabla \tilde{u} - \tilde{f}\big)
  \cdot \nabla (\varphi - \overline{\phi})\dx.
\end{align*}
Using \eqref{eq:mirror} again,  we get
\begin{align*}
  \int_{B_{r_0}} \big(\tilde{A} \nabla \tilde{u} - \tilde{f}\big)
  \cdot \nabla \varphi\dx &=- \int_{B_{r_0}^+} \big(A \nabla u -
  f\big) \cdot \nabla \big( (\varphi - \overline{\phi}) \circ
  T^{-1}\big) \dx.
\end{align*}
Since $(\varphi - \overline{\phi})
\circ T^{-1} = 0$ on $\partial \Omega$, we finally deduce with the help of~\eqref{wfL2} that
\begin{align*}
  \int_{B_{r_0}} \big(\tilde{A} \nabla \tilde{u} - \tilde{f}\big)
  \cdot \nabla \varphi\dx &= 0
\end{align*}
for all $\phi \in \mathcal{C}^{0,1}_0(B_{r_0};\setR^N)$, which
proves~\eqref{extend}.

Consequently, we see that \eqref{extend} holds and therefore we shall apply the local result stated in Lemma~\ref{thm:ell-local}. To do so, we need to check the assumptions. First, the ellipticity of $\tilde{A}$ can be shown directly from the definition and the fact that $J$ is a regular matrix. Moreover, the constants of ellipticity of $\tilde{A}$ depends only on the same constant for $A$ and on the shape of $\Omega$. Further, to be able to use \eqref{unlocal}, we need to show small oscillations of $\tilde{A}$. Since, $T$ is $\mathcal{C}^1$ we have
$$
\begin{aligned}
\sup_{x,y \in B_{r_0}^{-}}|\tilde{A}(x)-\tilde{A}(y)|&\le \sup_{x,y \in B_{r_0}^{+}}|J(x)A(x)J^T(x)-J(y)A(y)J^T(y)|\\
&\le C\sup_{x,y \in B_{r_0}^{+}}|A(x)-A(y)| + C\sup_{x,y \in B_{r_0}^{+}}|J(x)-J(y)|.
\end{aligned}
$$
Similarly, we can also deduce that
$$
\begin{aligned}
\sup_{x\in B_{r_0}^{-}, y\in B_{r_0}^{+}}|\tilde{A}(x)-\tilde{A}(y)|&\le \sup_{x,y \in B_{r_0}^{+}}|J(x)A(x)J^T(x)-A(y)|\\
&\le C\sup_{x,y \in B_{r_0}^{+}}|A(x)-A(y)| + C\sup_{x\in B_{r_0}^{+}}|J(x)A(x)J^T(x)-A(x)|\\
&\le C\sup_{x,y \in B_{r_0}^{+}}|A(x)-A(y)| + C\sup_{x\in B_{r_0}^{+}}|\nabla a(x')|.
\end{aligned}
$$
Therefore, due to the continuity of $J$ and the fact that $\nabla a(0)=0$, we see that for any $\delta>0$ we can find $r^*>0$ such that
$$
C\sup_{x,y \in B_{r^*}^{+}}|J(x)-J(y)|+ C\sup_{x\in B_{r^*}^{+}}|\nabla a(x')| < \frac{\delta}{2}.
$$
Thus, assuming that
$$
\sup_{x,y\in \Omega; \, C|x-y|\le r^*} |A(x)-A(y)|\le \frac{\delta}{2}
$$
we can conclude that
$$
\sup_{x,y \in B_{r^*}}|\tilde{A}(x)-\tilde{A}(y)|\le \delta.
$$
We find $\delta>0$ and fix $r^*$ such that all assumptions of Lemma~\ref{thm:ell-local} are satisfied and  we consequently have
\begin{equation*}
\begin{split}
\bigg(\dashint_{B_{r^*}(x_0)}\abs{\nabla \tilde{u}}^p\tilde{\omega} \dx\bigg)^\frac1p&\leq  C\bigg( \dashint_{B_{2r^*}(x_0)}\abs{\tilde{f}}^p\tilde{\omega} \dx\bigg)^\frac1p+ C \bigg(\dashint_{B_{2r^*}(x_0)}\tilde\omega \dx\bigg)^{\frac{1}{p}}\bigg(\dashint_{B_{2r^*}(x_0)}\abs{\nabla \tilde{u}}^{\tilde{q}} \dx\bigg)^{\frac{1}{\tilde{q}}}
\end{split}
\end{equation*}
and \eqref{uptobound} follows directly.

\bigskip

\subsection{Global estimates:}
Finally, we focus on the proof of Theorem~\ref{thm:ell}. Recall that
the ball~$Q_0$ is a superset of~$\Omega$. Since $A$ is continuous, we
can find for any $\delta>0$ some $r^*$ such that
$$
\sup_{x,y\in \Omega;\, |x-y|\le r^*}|A(x)-A(y)|\le \delta.
$$
Therefore on any sufficiently small ball, we can use the estimate
\eqref{uptobound}. Since~$\Omega$
has~$\mathcal{C}^1$ boundary, we can find a finite covering of $\Omega$ by
balls $B_i$ of radii at most equal to $r^*$ such that  $\abs{B_i \cap
  \Omega} \geq c\, \abs{B_i}$. Then it follows from \eqref{uptobound}
and~\eqref{stth-new} that
\begin{equation*}
\begin{split}
  &\int_{\Omega}\abs{\nabla u}^p\omega \dx \leq
  C\int_{\Omega}\abs{f}^p\omega \dx+ C\sum_i
  \frac{\omega(2B_i)}{|2B_i|^{\frac{p}{\tilde{q}}}}
  \left(\int_{\Omega}\abs{\nabla u}^{\tilde{q}}
    \dx\right)^{\frac{p}{\tilde{q}}}
  \\
  &\leq C\int_{\Omega}\abs{f}^p\omega \dx+
  C(p,\tilde{q},A,\Omega)\,\omega(Q_0)\left(\int_{\Omega}\abs{\nabla
      u}^{\tilde{q}} \dx\right)^{\frac{p}{\tilde{q}}}
  \leq C(A,\Omega, A_p(\omega))\,
  \int_{\Omega}\abs{f}^p\omega \dx,
\end{split}
\end{equation*}
which finishes the proof of Theorem~\ref{thm:ell}.

\section{Proof of Theorem~\ref{T5}}
We start the proof by observing that \eqref{bit3} leads to the estimate
$$
\int_{\Omega} |a^k \cdot b^k |\omega \dx \le \int_{\Omega} \abs{a^k}^p \omega + \abs{b^k}^{p'}\omega \le C.
$$
Consequently, we can use Lemma~\ref{thm:blem} to conclude that there is a non-decreasing sequence of measurable sets $E_j \subset \Omega$ fulfilling $|\Omega \setminus E_j| \to 0$ as  $j \to \infty$ such that for any $j\in \mathbb{N}$ and any $\varepsilon >0$ there exists a $\delta >0$ such that for each $U\subset E_j$ fulfilling $|U|\le \delta$ there holds
\begin{equation}\label{prbit1}
\sup_{k\in \mathbb{N}}\int_{U}|a^k \cdot b^k |\omega \dx \le \sup_{k\in \mathbb{N}} \int_{U}\abs{a^k}^p \omega + \abs{b^k}^{p'} \omega \dx  \le \varepsilon.
\end{equation}
Consequently, for any $E_j$ we can extract a subsequence that we do not relabel such that
\begin{equation}\label{chcic}
  a^k \cdot b^k \omega \rightharpoonup \overline{a\cdot b\, \omega} \qquad \textrm{ weakly in } L^1(E_j),
\end{equation}
where $\overline{a\cdot b\, \omega}$ denotes in our notation the weak
limit. Further, since $L^p_\omega(\Omega)$ and $L^{p'}_\omega(\Omega)$ are reflexive,
we can pass to a (non-relabeled) subsequence with
\begin{align}
\label{bit}
\begin{aligned}
  a_k \weakto a \qquad \text{weakly in $L^p_\omega(\Omega;\setR^n)$},
  \\
  b_k \weakto b \qquad \text{weakly in $L^{p'}_\omega(\Omega;\setR^n)$}.
\end{aligned}
\end{align}
Our goal is to show that
\begin{equation}
\overline{a\cdot b \, \omega} = a\cdot b \, \omega \textrm{ almost everywhere in } \Omega. \label{prbit2}
\end{equation}
Indeed, if this is the case then it follows that not only a subsequence but the whole sequence fulfills \eqref{chcic}.

Since $\omega \in \mathcal{A}_p$, we can find by~\eqref{eq:lqprop}
some $q>1$ such that $L^p_\omega(\Omega) \embedding L^q(\Omega)$. This implies
\begin{align}
a^k &\rightharpoonup a &&\textrm{weakly in } L^q(\Omega; \setR^n).\label{bit1b}
\end{align}
Moreover, since the mapping $g \mapsto g \omega^{\frac 1s}$ is an
isometry from $L^s_\omega(\Omega)$ to $L^s(\Omega)$, we also have
\begin{align}
  a^k \omega^{\frac{1}{p}} &\rightharpoonup a \omega^{\frac{1}{p}} &&\textrm{weakly in } L^p(\Omega; \setR^n),\label{bit1str}\\
  b^k \omega^{\frac{1}{p'}} &\rightharpoonup b \omega^{\frac{1}{p'}}
  &&\textrm{weakly in } L^{p'}(\Omega; \setR^n)\label{bit2str}.
\end{align}
Then, extending $a^k$ by zero outside $\Omega$ we can introduce $d^k$ such that
$$
\triangle d^k = a^k \qquad \textrm{ in }\setR^n,
$$
i.e., we set $d^k:= a^k * G$, where $G$ denotes the Green function of the Laplace operator on the whole $\setR^n$. Then, using \eqref{bit1b} we see that
\begin{equation}
d^k \rightharpoonup d \qquad \textrm{ weakly in } W^{2,q}_{\loc}(\setR^n; \setR^n),\label{dconb}
\end{equation}
where
$$
\triangle d = a \qquad \textrm{ in }\setR^n.
$$
In addition, using \eqref{bit3} and the weighted theory for Laplace equation on $\setR^n$ (see \cite[p.244]{CoiFef74}) we can deduce
\begin{equation}
\nabla^2 d^k \rightharpoonup \nabla^2 d \qquad \textrm{ weakly in } L^{p}_{\omega}(\setR^n; \setR^{n\times n \times n}).\label{dconb2}
\end{equation}
Hence, to show \eqref{prbit2}, it is enough to check whether
\begin{align}
b^k \cdot (a^k -\nabla \divergence d^k)\omega &\rightharpoonup b \cdot (a -\nabla \divergence d)\omega &&\textrm{weakly in } L^1(E_j),\label{wan1}\\
b^k \cdot \nabla (\divergence d^k)\omega &\rightharpoonup  b \cdot \nabla (\divergence d) \omega &&\textrm{weakly in } L^1(E_j) \label{wan2}
\end{align}
for all $j \in \mathbb{N}$.

First, we focus on \eqref{wan1}. Assume for a moment that we know
\begin{equation}
\lim_{k\to \infty} \int_{\Omega} |a^k-a + \nabla (\divergence (d-d^k))|\tau \dx = 0 \label{str1a}
\end{equation}
for all nonnegative $\tau \in \mathcal{D}(\Omega)$. Then for any $\varphi \in L^{\infty}(E_j)$ we have
$$
\begin{aligned}
&\lim_{k\to \infty}\int_{E_j}b^k \cdot (a^k -\nabla \divergence d^k)\omega \varphi \dx \\
&=\lim_{k\to \infty}\int_{E_j}b^k \cdot (a -\nabla \divergence d)\omega \varphi \dx +\lim_{k\to \infty}\int_{E_j}b^k \cdot (a^k-a +\nabla \divergence (d- d^k))\omega \varphi \dx\\
&\overset{\eqref{bit2str}}=\int_{E_j}b \cdot (a -\nabla \divergence d)\omega \varphi \dx +\lim_{k\to \infty}\int_{E_j}b^k \cdot (a^k-a +\nabla \divergence (d- d^k))\omega \varphi \dx
\end{aligned}
$$
and \eqref{wan1} follows provided that the second limit in the above formula vanishes. However, we first notice that (for a subsequence) \eqref{str1a} implies that
\begin{equation}
\label{seqml}
b^k \cdot (a^k-a +\nabla \divergence (d- d^k))\omega \varphi \to 0 \textrm{ a.e. in } \Omega.
\end{equation}
Second, using \eqref{dconb} and \eqref{bit1str} we see that for any $U\subset E_j$ we have
$$
\int_{U}|b^k \cdot (a^k-a +\nabla \divergence (d- d^k))\omega \varphi| \dx\le C\|\varphi\|_{\infty}\|a^k-a\|_{L^p_{\omega}(\Omega)} \left(\int_{U}\abs{b^k}^{p'}\omega \dx \right)^{\frac{1}{p'}}.
$$
Then the equi-integrability \eqref{prbit1} also guarantees the equi-integrability  of the sequence \eqref{seqml} and consequently, the Vitali theorem leads to
$$
\lim_{k\to \infty}\int_{E_j}b^k \cdot (a^k-a +\nabla \divergence (d- d^k))\omega \varphi \dx=0,
$$
which finishes the proof of \eqref{wan1} provided we show \eqref{str1a}. First, it follows from \eqref{bit5} and \eqref{bit1b} that for a subsequence that we do  not relabel $\partial_{x_i}a_j^k-\partial_{x_j}a_i^k\to\partial_{x_i}a_j-\partial_{x_j}a_i$ strongly in $(W^{1,r}_0(\Omega))^*$ for all $i,j=1,\ldots, n$. Therefore, by the regularity theory for Poisson's equation we find that
\begin{equation}\label{kasl}
\partial_{x_i}d^k_j - \partial_{x_j}d^k_i \to \partial_{x_i}d_j - \partial_{x_j}d_i \textrm{ strongly in } W^{1,r}_{\loc}(\Omega)
\end{equation}
for all $i,j=1,\ldots n$ and all $r\in [1,q)$, where $q>1$ comes from \eqref{bit1b}. Moreover, using the definition of $d^k$ we have
$$
a^k_j -\partial_{x_j} \divergence d^k = \sum_{m=1}^n \partial^2_{x_m^2} d^k_j - \partial_{x_j}\partial_{x_m} d^k_m = \sum_{m=1}^n\partial_{x_m}\left(\partial_{x_m} d^k_j - \partial_{x_j} d^k_m\right)
$$
and with the help of \eqref{kasl} we see that \eqref{str1a} directly follows and the proof of \eqref{wan1} is complete.

The rest  of this subsection is devoted to the most difficult part of the proof, which is the validity of \eqref{wan2}. For simplicity, we denote $e^k:=\divergence d^k$ and due to \eqref{dconb} and \eqref{dconb2} we have that
\begin{align}
e^k &\rightharpoonup e &&\textrm{ weakly in } W^{1,q}_{\loc}(\setR^n),\label{cek}\\
\nabla e^k &\rightharpoonup \nabla e &&\textrm{ weakly in } L^{p}_{\omega}(\setR^n; \setR^n),\label{cek2}
\end{align}
where $e=\divergence d$. Since we are interested only in the convergence result in $\Omega$ we localize $e^k$ by a proper cutting outside $\Omega$. To be more precise on the ball $B$ (recall that it is a ball such that $\Omega \subsetneq B$) we set
$$
e^k_B:=e^k \tau
$$
with $\tau \in \mathcal{D}(B)$ being identically one in $\Omega$. In addition, we can observe that
\begin{align}
e^k_B &\rightharpoonup e_B &&\textrm{ weakly in } W^{1,q}_{0}(B),\label{cekB}\\
\nabla e^k_B &\rightharpoonup \nabla e_B &&\textrm{ weakly in } L^{p}_{\omega}(B; \setR^n).\label{cek2B}
\end{align}
Indeed, the relation \eqref{cekB} it a trivial consequence of \eqref{cek} and for the validity of \eqref{cek2B} it is enough to show that
$$
\int_B |\nabla e^k_B|^p \omega \dx \le C.
$$
Since $|\nabla e^k_B|\le C|\nabla e^k| + C|e^k-(e^k)_B| + C|(e^k)_B|$, where $e^k_B$ denotes the mean value of $e^k$ over $B$, it follows form \eqref{cek} and \eqref{cek2}  that we just need to estimate the term involving $|e^k-(e^k)_B|$. But using the point-wise estimate $|e^k-(e^k)_B|\le C(B)M(\nabla e^k)$ we have
$$
\int_B |e^k_B-(e^k)_B|^p \omega \dx \le  C \int_{\setR^n} |M(\nabla e^k)|^p \omega \dx \le C A_p(\omega)\int_{\setR^n} |\nabla e^k|^p \omega \dx \le C,
$$
where we used the properties of $\Acal_p$-weights. Finally, since $e^k_B\in W^{1,1}_0(B)$ we can apply the Lipschitz approximation (Theorem~\ref{thm:liptrunc}), which implies that for arbitrary fixed $\lambda>\lambda_0$ and for any $k$ we find the Lipschitz approximation of $e^k_B$ on the set $B$ and denote it by $e^{k,\lambda}_B$. Then thanks to Theorem~\ref{thm:liptrunc}, for any $\lambda$ we can find a subsequence (that is not relabeled) such that
\begin{align}
\nabla e^{k,\lambda}_B &\rightharpoonup^* \nabla e^{\lambda}_B &&\textrm{ weakly$^*$ in } L^{\infty}(B;\setR^n),\label{cekBL}\\
\nabla e^{k,\lambda}_B &\rightharpoonup \nabla e^{\lambda}_B &&\textrm{ weakly in } L^{p}_{\omega}(B; \setR^n),\label{cek2BL}\\
e^{k,\lambda}_B &\to e^{\lambda}_B &&\textrm{ strongly in } \mathcal{C}(B).\label{cek3BL}
\end{align}
Please notice that we do not have any a~priori knowledge of how the limit $e^{\lambda}_B$ can be found, we just know that it exists.

In the next step, we identify the weak limit of $b^k \cdot \nabla e^{k,\lambda}_B$. Due to \eqref{bit} and \eqref{cekBL}, we see that this sequence is equi-integrable and consequently poses a weakly converging (in the topology of $L^1$) subsequence. Therefore, to identify it, it is enough to show that for all $\eta \in \mathcal{D}(\Omega)$ we have
$$
\lim_{k\to \infty}\int_{\Omega} b^k \cdot \nabla e^{k,\lambda}_B \eta \dx = \int_{\Omega} b \cdot \nabla e^{\lambda}_B \eta \dx.
$$
However, using \eqref{bit4}, \eqref{cekBL} and \eqref{cek3BL} we can deduce that
$$
\begin{aligned}
&\lim_{k\to \infty}\int_{\Omega} b^k \cdot \nabla e^{k,\lambda}_B \eta \dx = \lim_{k\to \infty}\int_{\Omega} b^k \cdot (\nabla e^{k,\lambda}_B-\nabla e^{\lambda}_B) \eta \dx+\int_{\Omega} b \cdot \nabla e^{\lambda}_B \eta \dx\\
&= \lim_{k\to \infty}\int_{\Omega} b^k \cdot \nabla ((e^{k,\lambda}_B- e^{\lambda}_B) \eta) \dx+ \lim_{k\to \infty}\int_{\Omega} b^k \cdot  \nabla\eta (e^{k,\lambda}_B- e^{\lambda}_B) \dx +\int_{\Omega} b \cdot \nabla e^{\lambda}_B \eta \dx\\
&=\int_{\Omega} b \cdot \nabla e^{\lambda}_B \eta \dx
\end{aligned}
$$
and therefore we have
\begin{equation}
b^k \cdot \nabla e^{k,\lambda}_B  \rightharpoonup b \cdot \nabla e^{\lambda}_B \qquad \textrm{ weakly in } L^1(\Omega). \label{Lweak}
\end{equation}
Finally, let $\varphi \in L^{\infty}(E_j)$ be arbitrary and $C:=C(|\varphi\|_{\infty})$. Then we check the validity of \eqref{wan2} as follow.
\begin{equation}\label{zaklad}
\begin{aligned}
&\lim_{k\to \infty}\left|\int_{E_j}(b^k \cdot \nabla (\divergence d^k)- b \cdot \nabla (\divergence d))\omega \varphi \dx\right|  =\lim_{k\to \infty}\left|\int_{E_j}(b^k \cdot \nabla e^{k}_B- b \cdot \nabla e_B)\omega \varphi \dx\right|\\
&\le \lim_{k\to \infty}\left|\int_{E_j}(b^k \cdot \nabla e^{k,\lambda}_B- b \cdot \nabla e^{\lambda}_B)\omega \varphi \dx\right|+C\limsup_{k\to \infty}\int_{E_j}|b^k| |\nabla (e^k_B-e^{k,\lambda}_B)|\omega\dx\\
&\quad+\left|\int_{E_j}b \cdot \nabla (e_B-e^{\lambda}_B)\omega \varphi \dx\right|\\
&\le \lim_{k\to \infty}\left|\int_{E_j}\frac{(b^k \cdot \nabla e^{k,\lambda}_B- b \cdot \nabla e^{\lambda}_B)\varphi \omega}{1+\varepsilon\omega} \dx\right|+C\limsup_{k\to \infty}\left|\int_{E_j} \frac{\varepsilon \omega^2(\abs{b^k}\abs{\nabla e^{k,\lambda}_B}+ \abs{b} \abs{\nabla e^{\lambda}_B}) }{1+\varepsilon \omega}\dx\right|\\
&\quad+C\limsup_{k\to \infty}\int_{E_j}|b^k| |\nabla (e^k_B-e^{k,\lambda}_B)|\omega\dx+\left|\int_{E_j}b \cdot \nabla (e_B-e^{\lambda}_B)\omega \varphi \dx\right|\\
&\leq C\limsup_{k\to \infty}\int_{E_j}\frac{\varepsilon \omega^2 \abs{b^k}\abs{\nabla e^{k,\lambda}_B}}{1+\varepsilon \omega}+C\limsup_{k\to \infty}\int_{E_j}|b^k| |\nabla (e^k_B-e^{k,\lambda}_B)|\omega \dx\\
&\qquad +\left|\int_{E_j}b \cdot \nabla (e_B-e^{\lambda}_B)\omega \varphi \dx\right|+C\int_{E_j}\frac{\varepsilon \omega^2 \abs{b} \abs{\nabla e^{\lambda}_B} }{1+\varepsilon \omega}\dx=:(I)+(II)+(III)+(IV),
\end{aligned}
\end{equation}
where the last identity follows from \eqref{Lweak} since $\varphi \omega/(1+\varepsilon \omega)$ is a bounded function whenever $\varepsilon >0$. In the next step, we show the all terms on the right hand side vanishes when we let $\varepsilon \to 0_+$ and $\lambda \to \infty$. To do so, we first observe that thanks to Theorem~\ref{thm:liptrunc} and the weak lower semicontinuity  we have
\begin{align}
\nabla e^{k,\lambda}_B \rightharpoonup \nabla e^{\lambda}_B &\qquad \textrm{ weakly in } L^p_{\omega}(\Omega; \setR^n),\label{ukilo}\\
e^{k,\lambda}_B \rightharpoonup  e^{\lambda}_B &\qquad \textrm{ weakly in } W^{1,q}(\Omega),\label{ukilo1}\\
\int_{\Omega} |\nabla e^{\lambda}_B|^q+|\nabla e^{\lambda}_B|^p \omega \dx &\le C\liminf_{k\to \infty}\int_{B} |\nabla e^{k}_B|^q+|\nabla e^k_B|^p\omega \dx \le C. \label{ukiyo2}
\end{align}
Therefore, applying the H\"{o}lder inequality, we have the estimate
$$
\int_{E_j}\abs{b} \abs{\nabla e^{\lambda}_B} \omega \dx \le C.
$$
Consequently, using the Lebesgue dominated convergence theorem (and also the fact that $\omega$ is finite almost everywhere), we deduce
\begin{equation}
\label{lim1}
\lim_{\varepsilon \to 0_+}(IV)=C\lim_{\varepsilon \to 0_+}\int_{E_j}\abs{b} \abs{\nabla e^{\lambda}_B} \frac{\varepsilon \omega^2 }{1+\varepsilon \omega}\dx=0.
\end{equation}
For the second term involving $\varepsilon$ the key property is the uniform equi-integrability of $b^k$ stated in \eqref{prbit1}. Indeed, applying the H\"{o}lder inequality and \eqref{ukiyo2} we have
\begin{align*}
\lim_{\varepsilon \to 0_+}(I)&=C\limsup_{\varepsilon \to 0_+}\limsup_{k\to \infty}\int_{E_j}\abs{b^k}\abs{\nabla e^{k,\lambda}_B}\frac{\varepsilon \omega^2 |\varphi|}{1+\varepsilon \omega} \dx \\
&\le C\limsup_{\varepsilon \to 0_+}\limsup_{k\to \infty}\left(\int_{E_j}\abs{b^k}^{p'}\omega \frac{\varepsilon \omega}{1+\varepsilon \omega} \dx\right)^{\frac{1}{p'}}\left(\int_{E_j}\abs{\nabla e^{k,\lambda}_B}^p\omega \dx\right)^{\frac{1}{p}}\\
&\le C\limsup_{\varepsilon \to 0_+}\limsup_{k\to \infty}\left(\int_{E_j\cap\{w>\lambda\}}\abs{b^k}^{p'}\omega \frac{\varepsilon \omega}{1+\varepsilon \omega} \dx\right)^{\frac{1}{p'}}\\
&\quad+C\limsup_{\varepsilon \to 0_+}\limsup_{k\to \infty}\left(\int_{E_j\cap\{w\le \lambda\}}\abs{b^k}^{p'}\omega \frac{\varepsilon \omega}{1+\varepsilon \omega} \dx\right)^{\frac{1}{p'}}\\
&  \le C\limsup_{k\to \infty}\left(\int_{E_j\cap\{w>\lambda\}}\abs{b^k}^{p'}\omega \dx\right)^{\frac{1}{p'}}.
\end{align*}
Since $|\{\omega >\lambda\}|\le C/\lambda$, we can use \eqref{prbit1} and let $\lambda \to \infty$ in the last inequality to deduce
\begin{equation}
\label{lim2}
\begin{split}
&\limsup_{\lambda\to\infty}\limsup_{\varepsilon \to 0_+}\limsup_{k\to \infty}\int_{E_j}\abs{b^k}\abs{\nabla e^{k,\lambda}_B}\frac{\varepsilon \omega^2 |\varphi|}{1+\varepsilon \omega} \dx =0.
\end{split}
\end{equation}
Next, we let $\lambda \to \infty$ in all remaining terms on the right hand side of \eqref{zaklad}. Using \eqref{itm:weight} and the H\"{o}lder inequality, we have
\begin{equation}
\begin{split}
\limsup_{\lambda \to \infty}(II)&=C\limsup_{\lambda \to \infty}\limsup_{k\to \infty}\int_{E_j}|b^k| |\nabla (e^k_B-e^{k,\lambda}_B)|\omega \dx\\
&=C\limsup_{\lambda \to \infty}\limsup_{k\to \infty}\int_{E_j\cap\{M(\nabla e^k_B)>\lambda\}}|b^k| |\nabla (e^k_B-e^{k,\lambda}_B)|\omega \dx\\
&\le C\limsup_{\lambda \to \infty}\limsup_{k\to \infty}\left(\int_{E_j\cap\{M(\nabla e^k_B)>\lambda\}}\abs{b^k}^{p'}\omega \dx\right)^{\frac{1}{p'}}=0,
\end{split}
\label{lim4}
\end{equation}
where the last inequality follows from the fact that $|\{M(\nabla e^k_B)>\lambda\}|\le C/\lambda$ and \eqref{prbit1}. Finally, we are left to show
\begin{equation}
\begin{split}
\lim_{\lambda \to \infty}(III)=&\lim_{\lambda \to \infty}\left|\int_{E_j}b \cdot \nabla (e_B-e^{\lambda}_B)\omega \varphi \dx\right|=0.
\end{split}
\label{lim5}
\end{equation}
However, to get \eqref{lim5}, it is enough to show that
$$
\nabla e^{\lambda}_B \rightharpoonup \nabla e_B \qquad \textrm{weakly in } L^p_{\omega}(\Omega; \setR^n).
$$
Due to \eqref{ukiyo2}, we however have that there is some $\overline{e_B}\in W^{1,q}(\Omega)$ such that
$$
\begin{aligned}
e^{\lambda}_B &\rightharpoonup \overline{e_B} &&\textrm{weakly in } W^{1,q}(\Omega),\\
\nabla e^{\lambda}_B &\rightharpoonup \nabla \overline{e_B} &&\textrm{weakly in } L^p_{\omega}(\Omega; \setR^n).
\end{aligned}
$$
Hence, due to the uniqueness of the weak limit it is enough to check that $\overline{e}_B=e_B$. To do so, we use the compact embedding $W^{1,1}(\Omega) \embedding \embedding L^1(\Omega)$ to get
$$
\begin{aligned}
\|\overline{e}_B-e_B\|_1 &= \lim_{\lambda \to \infty}\int_{\Omega}|e^{\lambda}_B - e_B|\dx = \lim_{\lambda \to \infty} \lim_{k\to \infty}\int_{\Omega}|e^{k,\lambda}_B - e^k_B|\dx\\
&=\lim_{\lambda \to \infty} \lim_{k\to \infty}\int_{\Omega\cap\{M(\nabla e^k_B)>\lambda\}}|e^{k,\lambda}_B - e^k_B|\dx\\
&\le \lim_{\lambda \to \infty} \lim_{k\to \infty} \|e^{k,\lambda}_B - e^k_B\|_q |\Omega\cap\{M(\nabla e^k_B)>\lambda\}|^{\frac{1}{q'}}\le C\lim_{\lambda \to \infty}\lambda^{-\frac{1}{q'}}=0
\end{aligned}
$$
and consequently \eqref{lim5} holds. Hence, using \eqref{lim1}--\eqref{lim5} in \eqref{zaklad}, we deduce \eqref{wan2} and the proof is complete.

\section{Proof of Theorem~\ref{thm:liptrunc}}

This part of the paper is devoted to the proof of
Theorem~\ref{thm:liptrunc}. All statements except~\eqref{itm:weight}
are already contained in \cite[Theorem~13]{DieKreSul12} (see
also~\cite{Die13paseky} for a survey on the Lipschitz truncation). The
first inequality of~\eqref{itm:weight} follows directly from the
second one, so it is enough to prove the second estimate.

It follows from~\eqref{eq:lip1} and \eqref{eq:lip2} that
\begin{align*}
  \norm{\nabla (g-g^\lambda)}_{L^p_\omega}&\leq \norm{\nabla
    (g-g^\lambda) \chi_{\set{M(\nabla g)>\lambda}}}_{L^p_\omega}
  \\
  &\leq \norm{\nabla g\, \chi_{\set{M(\nabla g)>\lambda}}}_{L^p_\omega}
  + c\, \norm{\lambda\, \chi_{\set{M(\nabla g)>\lambda}}}_{L^p_\omega}.
\end{align*}
We need to control the second term in the last estimate.
Let us consider the open set $\set{M(\nabla g)>\lambda}$. For every $x
\in \set{M(\nabla g)>\lambda}$ there exists a ball $B_{r(x)}(x)$ with
\begin{align}
  \label{eq:lip5}
  \lambda< \dashint_{B_r(x)}\abs{\nabla g} dx\leq 2\lambda.
\end{align}
These balls cover $\set{M(\nabla g)>\lambda}$.  Next, using Besicovich
covering theorem, we can extract from this cover a countable subset
$B_i$ which is locally finite, i.e.
\begin{align}
  \label{eq:lip6}
  \# \{j\in \mathbb{N}; \, B_i\cap B_j \neq \emptyset\} \le C(n).
\end{align}
Using~\eqref{eq:lip5} and \eqref{eq:lip6} we have the estimate
\begin{align*}
  &\norm{\lambda\, \chi_{\set{M(\nabla g)>\lambda}}}_{L^p_\omega}^p =
  \lambda^p \omega(\set{M(\nabla g)>\lambda}) \leq \sum_i \lambda^p
  \omega(B_i)
  \\
  &\quad\leq \sum_i \bigg( \dashint_{B_i} \abs{\nabla g}\,dx \bigg)^p
  \omega(B_i)
  \leq \sum_i \dashint_{B_i} \abs{\nabla g}^p \omega\,dx
  \bigg(\dashint_{B_i} \omega^{-(p'-1)}\,dx \bigg)^{\frac{1}{p'-1}}
  \omega(B_i)
  \\
  &\quad\leq \mathcal{A}_p(\omega) \sum_i \int_{B_i} \abs{\nabla g}^p
  \omega \,dx
	\leq C(n) \,\mathcal{A}_p(\omega) \int_{\set{M(\nabla g)> \lambda}}
  \abs{\nabla g}^p \omega \,dx.
\end{align*}
This directly leads to the following inequality
\begin{align*}
  \norm{\lambda\, \chi_{\set{M(\nabla g)>\lambda}}}_{L^p_\omega} &\leq
  C(n) \,\mathcal{A}_p(\omega)^{\frac 1p} \norm{\nabla g\,
    \chi_{\set{M(\nabla g)>\lambda}}}_{L^p_\omega},
\end{align*}
which proves the desired estimate~\eqref{itm:weight}.

\section{Proof of Theorem~\ref{T3D}}\label{Last:S}
We present only a sketch of the  proof here, since all steps were already justified in the proof of Theorem~\ref{T3}. Hence, to obtain the a~priori estimate \eqref{apriori3D}, we  observe that
\begin{equation*}
\begin{split}
\int_{\Omega} \tilde{A}(x)(\nabla u-\nabla u_0) \cdot \nabla \varphi \dx = \int_{\Omega} \left(f -\tilde{A}(x)\nabla u_0 +\tilde{A}(x)\nabla u -A(x,\nabla u)\right)\cdot \nabla \varphi \dx,
\end{split}
\end{equation*}
which by the  use of Theorem~\ref{thm:ell} (note here that $u-u_0$ has zero trace) and \eqref{ass:A} leads to
\begin{equation*}
\begin{split}
\int_{\Omega} |\nabla u-\nabla u_0|^p\omega \dx &\le C\int_{\Omega} \left(|f|^p+|\nabla u_0|^p +|\tilde{A}(x)\nabla u -A(x,\nabla u)|^p\right)\omega \dx\\
&\le C(\varepsilon)\int_{\Omega} \left(|f|^p+|\nabla u_0|^p +1\right)\omega \dx+ \varepsilon \int_{\Omega} |\nabla u|^p\omega \dx.
\end{split}
\end{equation*}
Consequently, choosing $\varepsilon$ small enough and using the triangle inequality, we find \eqref{apriori3}. The existence is then identically the same, we just also need to approximate $u_0$ by a sequence of smooth functions such that
$$
u^k_0 \to u_0 \qquad \textrm{strongly in } W^{1,\tilde{q}}(\Omega; \setR^N).
$$
Finally, for the uniqueness proof, we use the similar procedure and we see that if $u_1$ and $u_2$ are two solutions then
\begin{equation*}
\begin{split}
\int_{\Omega} \tilde{A}(x)(\nabla u_1-\nabla u_2) \cdot \nabla \varphi \dx = \int_{\Omega} \left(\tilde{A}(x)(\nabla u_1-\nabla u_2) +A(x,\nabla u_1) -A(x,\nabla u_2)\right)\cdot \nabla \varphi \dx
\end{split}
\end{equation*}
and  since $u_1-u_2 \in W^{1,\tilde{q}}_0(\Omega; \setR^N)$, we may now follow step by step the proof of Theorem~\ref{T3}.

\section{Proofs of Corollaries}\label{P:C}

\subsection{Proof of Corollary~\ref{CMB}}
The proof of Corollary~\ref{CMB} is rather straight forward. Indeed, for a given measure  $f\in \mathcal{M}(\Omega; \setR^N)$ we can use the classical theory and find $v\in W^{1,n'-\varepsilon}_0(\Omega; \setR^N)$ for all $\varepsilon >0$ solving
$$
\int_{\Omega}\nabla v \cdot \nabla \varphi \dx = \langle f, \varphi \rangle \qquad \textrm{for all } \varphi\in \mathcal{C}^{0,1}_0(\Omega; \setR^N).
$$
Then, it follows that $u$ is a solution to \eqref{WFGMB} if and only if it solves
\begin{equation}\label{WFGMB2}
\int_{\Omega}A(x,\nabla u) \cdot \nabla \phi \dx = \int_{\Omega} \nabla v \cdot \nabla \varphi\dx  \qquad \textrm{ for all } \phi\in \mathcal{C}^{0,1}_0(\Omega; \setR^N).
\end{equation}
Thus, we can now apply Theorem~\ref{T1} with $f:=\nabla v$ and all statements in Corollary~\ref{CMB} directly follows.

\subsection{Proof of Corollary~\ref{C2}}
\label{ssec:proofT2}
We show that Corollary~\ref{C2} can be directly
proved by using Theorem~\ref{T3}.
 Indeed, by setting
$$
\omega:=(1+Mf)^{q-2}=(M(1+|f|))^{q-2},
$$
where we extended $f$ by zero outside $\Omega$,
we can use Lemma~\ref{cor:dual} to deduce that $\omega \in \Acal_2$ provided that $|q-2|<1$. Since $q\in (1,2)$, we always have $|q-2|<1$ and therefore $\omega\in \Acal_2$. Consequently, we can construct a solution $u$ according to Theorem~\ref{T3}. Next, using \eqref{apriori3} and the continuity of the maximal  function, we can deduce
$$
\begin{aligned}
\int_{\Omega} \frac{|\nabla u|^2}{(1+Mf)^{2-q}}\dx &= \int_{\Omega} |\nabla u|^2 \omega \dx \le C(A_2(\omega),\Omega)\left(1+\int_{\Omega}|f|^2\omega \dx\right)\\
&= C(A_2(\omega),\Omega)\left(1+\int_{\Omega}\frac{|f|^2}{(1+Mf)^{2-q}} \dx\right)\\
&\le C(A_2(\omega),\Omega)\left(1+\int_{\Omega}(Mf)^q \dx\right)\le C(f,\Omega,q)\left(1+\int_{\Omega}|f|^q \dx\right),
\end{aligned}
$$
which is nothing else than \eqref{apriori2}.
\bibliographystyle{abbrv} 
\bibliography{lars}
\end{document}